\documentclass[Threeside,11pt]{article}
\usepackage [latin1]{inputenc}
\usepackage[english]{babel}
\usepackage{amsmath}
\usepackage{amsthm}
\usepackage{algorithm}
\usepackage{cite}
\usepackage{amsfonts}
\usepackage{amssymb}
\usepackage{mathrsfs}
\usepackage{graphicx}
\usepackage{colortbl,dcolumn}
\usepackage{marvosym}
\usepackage{ifsym}
\usepackage{paralist}
\usepackage{xcolor}
\usepackage{bm}
\usepackage[
colorlinks,
citecolor=blue,
linkcolor=blue,
backref=page
allcolors=black
]{hyperref} 
\allowdisplaybreaks

\newtheorem{lemma}{\bf Lemma}[section]
\newtheorem{theorem}{\bf Theorem}[section]
\newtheorem{proposition}{\bf Proposition}[section]
\newtheorem{corollary}{\bf Corollary}[section]
\newtheorem{definition}{\bf Definition}[section]
\newtheorem{remark}{\bf Remark}[section]

\newtheorem{assumption}{\bf Assumption}
\numberwithin{equation}{section}

\begin{document}
	
	\title{{\sl  Non-equilibrium fluctuations of a two-species exclusion process with slow boundary }}
\author{Ziying Chen $^{\mathcal{z},1}$, Yong Li $^{*,\mathcal{x},1,2}$ }

\renewcommand{\thefootnote}{}
\footnotetext{\hspace*{-6mm}
	
	\begin{tabular}{l l}
		$^{*}$~~~The corresponding author.\\
		$^\mathcal{z}$~~~E-mail address : ziying24@mails.jlu.edu.cn\\
		$^{\mathcal{x}}$~~~E-mail address : liyong@jlu.edu.cn\\
		$^{1}$~~~School of Mathematics, Jilin University, Changchun 130012, People's Republic of China.\\
		$^{2}$~~~School of Mathematics and Statistics, Center for Mathematics and Interdisciplinary Sciences,\\
		~~~ Northeast Normal University, Changchun 130024, People's Republic of China.
\end{tabular}}

\date{}
\maketitle

\begin{abstract}
 
		We study the non-equilibrium density fluctuations of a one-dimensional two-species symmetric simple exclusion process in contact with critical slow boundary reservoirs. The boundary rates are of order $1/n$, which leads to Robin boundary conditions at the macroscopic level.
	    We prove that the coupled fluctuation field associated with the empirical densities of two species converges to a generalized Ornstein-Uhlenbeck process. The limiting process is characterized by a linear martingale problem whose coefficients reflect the combined effects of bulk diffusion, species conversion, and the boundary reservoirs.
\end{abstract}

\section{  Introduction }

Interacting particle systems provide fundamental models for studying microscopic stochastic dynamics in statistical physics and probability theory. 
A central and challenging problem in this field is to study density fluctuations around the hydrodynamic limit. 
For equilibrium systems, fluctuations are usually studied around an equilibrium state. 
However, many systems describing the transport of mass or energy are continuously driven by external forces, boundary reservoirs, or internal reaction mechanisms. 
Such mechanisms typically break detailed balance and drive the system far from equilibrium. 
In particular, when the boundary reservoirs impose different densities, the resulting density gradient induces a persistent macroscopic current and leads to a non-equilibrium steady state. 
The study of non-equilibrium density fluctuations is therefore an important problem linking microscopic interacting particle systems with macroscopic non-equilibrium phenomena; see, e.g., \cite{ref11,ref12,ref13}.

The symmetric simple exclusion process is a classical model for studying these questions. 
In this model, particles perform symmetric random jumps on a lattice subject to the exclusion rule, namely, each site can be occupied by at most one particle. 
Despite its simple definition, the exclusion process has become a paradigmatic model for a variety of collective phenomena, including mass transport, interface growth, and motion by mean curvature. For the one-species symmetric simple exclusion process, the hydrodynamic limit and density fluctuations on the torus or in closed systems have been extensively investigated in \cite{ref14,ref15,ref16,ref17}.

When the system is in contact with external reservoirs, the creation and annihilation mechanisms at the boundaries place it in an open environment, and may lead to a non-equilibrium stationary state. 
In such open boundary systems, the strength of the reservoirs further affects the boundary conditions appearing in the macroscopic limit. A common way to tune the interaction with the reservoirs is to introduce a parameter $\theta\in\mathbb R$, so that the injection and removal rates at the boundaries are scaled by $n^{-\theta}$. 
When $\theta<0$, the reservoirs act fast, whereas when $\theta\ge 0$, the reservoirs are slow. 
This phenomenon was systematically investigated in \cite{ref2} for the symmetric simple exclusion process with slow boundaries, where it was shown that different choices of the boundary scaling parameter lead to Dirichlet, Robin, or Neumann boundary conditions at the macroscopic level. Subsequently, \cite{ref3,ref18} studied the non-equilibrium and stationary fluctuations of the one-species symmetric exclusion process in contact with slowed reservoirs. Further results on open boundary exclusion processes and their fluctuations can be found in \cite{ref19,ref20,ref21}.

However, many physical systems are intrinsically multi-component. 
For instance, in multi-component diffusion, biological transport, and surface reaction systems, different types of particles may not only move in space but also interact through conversion, competition, or cooperation mechanisms. 
Therefore, open multi-species particle systems provide natural microscopic models for studying the joint effects of multi-component transport, bulk reaction mechanisms, and boundary driving, rather than being merely formal extensions of one-species models. 
In such systems, inter-species interactions together with exclusion constraints may give rise, at the macroscopic level, to coupled diffusion equations or reaction-diffusion systems, while in the open-boundary setting the boundary reservoirs further affect the macroscopic boundary conditions and the corresponding fluctuation structure.

In recent years, hydrodynamic limits and fluctuation problems for multi-species interacting particle systems have been studied in various settings \cite{ref22,ref25,ref26,ref27}. 
Hydrodynamic limits and the corresponding macroscopic boundary conditions for multi-species reaction-diffusion models with open boundaries have also been investigated \cite{ref23,ref24}. However, for open non-equilibrium systems, the hydrodynamic limit only describes the deterministic evolution of the macroscopic density and is not sufficient to capture the random fluctuations induced by boundary driving. 
To understand how boundary driving and bulk species conversion jointly affect the limiting noise and covariance structure under exclusion constraints, it is necessary to study the corresponding density fluctuation problem. 
To the best of our knowledge, non-equilibrium fluctuations for multi-species symmetric exclusion processes with slow boundary reservoirs and bulk species conversion have not yet been systematically studied.

To address this problem, we consider a two-species symmetric simple exclusion process on a finite one-dimensional discrete interval in contact with slow boundary reservoirs. 
In this model, each site can accommodate at most one particle, and particles may belong to one of two species. 
The bulk dynamics consist of symmetric exchanges and species-conversion mechanisms, while at the boundaries the system is coupled to two external reservoirs through particle injection and removal. We focus on the critical slow-boundary regime, in which the reservoir rates are of order (1/n). 
At the macroscopic level, this boundary mechanism corresponds to Robin-type boundary conditions.

By constructing suitable Dynkin martingales, we prove that, in the critical slow-boundary regime, the coupled fluctuation field associated with two particle densities converges to a generalized Ornstein-Uhlenbeck process.  
The noise term has a composite structure: it contains both conservative fluctuations arising from the bulk exchange dynamics and non-conservative fluctuations generated by species conversion and by the random injection and removal of particles at the reservoirs. This result extends the existing fluctuation theory for one-species slow boundary models and provides a tractable mathematical framework for studying non-equilibrium fluctuations in more general multi-component open systems.

The main difficulties of this work arise from the coupling between two-species exclusion structure and the slow boundary mechanism. 
Since two species share the same exclusion constraint, the injection and removal mechanisms at the boundaries no longer have the scalar structure of the one-species case \cite{ref3}, but instead lead to coupled boundary dynamics. 
As a result, the treatment of boundary contributions and the identification of the Robin terms in the fluctuation limit become more delicate.
On the other hand, two-point correlation function is no longer scalar, but becomes a coupled system consisting of four correlation components.
In particular, the effective operator induced by the boundary generator no longer has a non-negative jump-rate structure on the off-diagonal components, and hence cannot be directly interpreted as the generator of a random walk. Therefore, the estimates based on diagonal occupation times of an auxiliary random walk, which are used in the one-species model, cannot be applied directly. New correlation estimates adapted to the coupled two-species system are required. 
Although the bulk species-conversion mechanism does not change the leading scaling structure, it  introduces additional coupling terms, which complicate the martingale decomposition and the computation of quadratic variations.

The paper is organized as follows. In Section 2, we give the precise definition of the two-species reaction-diffusion model with slow boundaries and state the main results, including the non-equilibrium fluctuation result in Theorem 2.2 and the characterization of the limiting Ornstein-Uhlenbeck process in Theorem 2.3. Section 3 is devoted to the basic properties of the limiting operator and its associated semigroup. In Section 4, we establish the replacement lemma needed for the proof of the main theorem. Sections 5 and 6 are devoted to the proofs of Theorems 2.2 and 2.3, respectively. Section 7 treats the tightness of the fluctuation fields. Finally, Section 8 collects several auxiliary estimates used throughout the proof, in particular the estimates on two-point correlation functions.

\section{Model and statement of results }

\subsection{Model description}

Let $ \Sigma_n=\{1,\cdots,n-1\} $, where \(n\in\mathbb N\) is the spatial scaling parameter. We consider a two-species exclusion process on \(\Sigma_n\). Each site \(x\in\Sigma_n\) can be occupied by at most one particle: it is either vacant, occupied by a particle of species \(A\), or occupied by a particle of species \(B\). We label the vacant state by \(0\), species \(A\) by \(1\), and species \(B\) by \(2\).
At each site \(x\in\Sigma_n\), we associate the occupation vector $ \bm{\eta}^x  = \left(  \eta_{0}^x ,\eta_{1}^x  ,\eta_{2}^x  \right)  $, where \(\eta_k^x \) denotes the occupation variable of state \(k\in\{0,1,2\}\) at site \(x\). Let $  \left( \bm{\eta}^{n}(t); t \ge 0 \right) $ be a Markov chain with state space
\begin{align*}
	\Omega_{n} = \left\lbrace \bm{\eta}: ~\bm \eta^{x} \in \left\lbrace 0,1 \right\rbrace ^{3}, ~\sum_{k=0}^{2} \eta_{k}^{x} = 1 ~\text{for all} ~ x \in \Sigma_n \right\rbrace.
\end{align*}

We now define the generator of the Markov process. For every function
\(f:\Omega_n\to\mathbb R\), the generator consists of four parts: the conservative exchange dynamics, the species conversion dynamics, and the left and right boundary dynamics.

The conservative exchange dynamics is defined by
\[
L_n^{s} f(\bm \eta) = \sum_{x=1}^{n-2} \sum_{k,l=0}^{2} \eta_k^x \eta_l^{x+1}  
\left[ f(\bm \eta -\delta_{k}^{x} + \delta_{l}^{x} + \delta_{k}^{x+1} - \delta_{l}^{x+1}) - f(\bm \eta) \right].
\]
It is interpreted as the exchange of label $ k $ at site $ x $ and label $l $ at site  $x+1$ with rate $ \eta_k^x \eta_l^{x+1} $ for every $ k,l \in \{ 0,1,2 \} $ and $ x\in \{ 1,\cdots,n-2 \}$.

The species conversion part acts only in the bulk and is defined by
\[
L_n^{c} f(\bm \eta) = \sum_{x=2}^{n-2} \sum_{k,l=1}^2 \gamma_k  \eta_k^x
\left[ f(\bm \eta -\delta_{k}^{x} + \delta_{l}^{x}) - f(\bm \eta) \right].
\] 
In other words, in the bulk region \(x\in \{2,\cdots,n-2\}\), a particle of species \(A\) is converted into a particle of species \(B\) at rate \(\gamma_1\), while a particle of species \(B\) is converted into a particle of species \(A\) at rate \(\gamma_2\).

The left boundary dynamics is defined by 
\[
L_n^{-} f(\bm \eta) = \sum_{k=1}^2 
\left(
\frac{\alpha_k}{n } \eta_0^1 \left[ f(\bm \eta - \delta_0^1 + \delta_k^1) - f(\bm \eta) \right]
+
\frac{ \alpha_0}{n } \eta_k^1 \left[ f(\bm \eta - \delta_k^1 + \delta_0^1) - f(\bm \eta) \right]
\right).
\]
It is interpreted as follows: at the left boundary site \(x=1\), a vacant site creates a particle of species \(k\in\{1,2\}\) at rate \(\alpha_k/n\), while a particle of either species is removed at rate \(\alpha_0/n\).

Similarly, the right boundary dynamics is defined by
\[
L_n^{+} f(\bm \eta) = \sum_{k=1}^2 
\left(
\frac{\beta_k}{n } \eta_0^{n-1} \left[ f(\bm \eta - \delta_0^{n-1} + \delta_k^{n-1}) - f(\bm \eta) \right]
+
\frac{ \beta_0}{n } \eta_k^{n-1} \left[ f(\bm \eta - \delta_k^{n-1} + \delta_0^{n-1} ) - f(\bm \eta) \right]
\right).
\]
Thus, at the right boundary site \(x=n-1\), a vacant site creates a particle of species \(k\in\{1,2\}\) at rate \(\beta_k/n\), while a particle of either species is removed at rate \(\beta_0/n\).

We restrict species conversion to the bulk sites in order to keep the boundary reservoirs as the sole mechanism acting at the endpoints. While extending the conversion dynamics to the two boundary sites would be a possible variant of the model, we do not include it here. This convention makes the connection between the reservoir dynamics and the resulting Robin boundary conditions fully transparent.

Throughout the paper, we assume that $ \gamma_1,\gamma_2 >0 $ and 
\[
\alpha_1,\alpha_2,\beta_1,\beta_2>0,\qquad
\alpha_1+\alpha_2<1,\qquad
\beta_1+\beta_2<1.
\]
Set $ \alpha_0 = 1- \alpha_1 - \alpha_2 $ and $ \beta_0 = 1- \beta_1 - \beta_2 $.
Then $  0< \alpha_0 <1  $ and $ 0< \beta_0 <1 $. 
For \(x,y\in\Sigma_n\) and \(k,l\in\{0,1,2\}\), the vector \(\delta_k^x\) is defined by
\begin{align*}
	\left ( \delta_{k}^{x}  \right )_{l}^{y}=\begin{cases}
		1, & \text{ if } ~ l=k ~\text{and} ~ y=x,  \\
		0, & \text{ otherwise}.
	\end{cases} 
\end{align*}

We define the infinitesimal generator of the Markov process as follows
\[
L_n = n^2 L_n^{s} + L_n^{c} + n^2 L_n^{-} + n^2 L_n^{+}.
\]

For a probability vector $ \bm p=(p_0,p_1,p_2)\in[0,1]^3 $ satisfying $ \sum_{i=0}^{2} p_{i} =1 $, we define the product measure \(\nu_{\bm p}^n\) on \(\Omega_n\) by 
\begin{align*}
	\nu _{\bm p}^{n} \left ( \bm{\eta} \right ) := \prod _{x =1}^{n-1} \prod_{i=0}^{2}  p_{i}^{\eta^{x}_{i}}.
\end{align*}
The measure  $  \nu _{\bm p}^{n} $ satisfies the detailed balance condition with respect to the conservative exchange dynamics $ L^s $ for every \(\bm p\). 
Moreover, if the two reservoirs impose the same density vector and the bulk conversion rates satisfy the detailed balance condition, namely
\[
\alpha_i =  \beta_i =  p_i,~ ~i=1,2
\qquad\text{and}\qquad
\gamma_1p_1=\gamma_2p_2,
\]
then \(\nu_{\bm p}^n\) is invariant for the full dynamics. In general, when the boundary reservoirs are different, the system is driven out of equilibrium. In that case, the invariant measure, denoted by \(\mu_{\mathrm{ss}}^n\), is no longer a product measure.

Before presenting the main results, we introduce some notation.
Let $ \mathrm I_m $ denote the $ m \times m $ identity matrix, and let $ M^{\top} $ denote the transpose of the matrix $ M $.
Let \(\mathcal C^\infty([0,1];\mathbb R)\) denote the space of real-valued smooth functions on \([0,1]\). For \(1\le p<\infty\), we denote by \(\mathcal L^p([0,1])\) the usual Lebesgue space. 
Let $ \mathcal D\left([0,T],\Omega_n\right) $ be the space of c\`adl\`ag trajectories from \([0,T]\) to \(\Omega_n\), endowed with the \(J_1\)-Skorokhod topology.
We denote by $ \mathbb{P} ^{n} $ the law of the process $ \left( \bm{\eta}^n \left( t\right) ; t \ge 0\right)  $ on the space $ \mathcal{D}\left ( [ 0,T ] , \Omega_{n} \right )  $, and by $ \mathbb{E}^{n} $ the corresponding expectation.
When the process $ \left( \bm{\eta}^n \left( t\right) ; t \ge 0\right)  $ starts from the initial law $ \mu^{n} $, its law and expectation are denoted respectively by $ \mathbb{P}_{\mu^{n}} ^{n} $ and $ \mathbb{E}_{\mu^{n}} ^{n} $.
Throughout this paper, we shall denote by $ C $ a finite positive constant, which may change from line to line.

\subsection{Hydrodynamic limit}

For each species \(k\in\{1,2\}\) and each test function 
\(f_k\in \mathcal C^\infty([0,1];\mathbb R)\), we define the empirical density field by
\begin{align*}
	X_{k}^{n,t}\left( \cdot\right) : ~&\mathcal{C}^{\infty }\left( [0,1];\mathbb{R}\right)  \to \mathbb{R} \notag \\
	& f_{k} \to \frac{1}{n} \sum_{x =1}^{n-1} \eta_{k}^{x}\left( t\right) f_{k}\left( \frac{x}{n}\right).
\end{align*}
 We first give the assumption on the behavior of the density field at the initial time.
\begin{definition}\label{d2.1}
	For $ k \in \{ 1,2 \} $, let $ \rho_k^0 :[0,1]\to[0,1] $ be a measurable density profile.  
	A sequence of probability measures \(\{\mu^n\}_{n\in\mathbb N}\) on \(\Omega_n\) is said to be associated with the initial macroscopic profile \(\bm\rho_0 = (\rho_1^0,\rho_2^0)^\top \) if, for every \(k\in\{1,2\}\), every \(\delta>0\), and every  \( f_k\in\mathcal C^\infty([0,1];\mathbb R) \),
	\[
	\lim_{n\to\infty}
	\mu^n\left(
	\left|
	X_k^{n,0}(f_k)
	-
	\int_0^1 f_k(u)\rho_k^0(u)\,\mathrm du
	\right|>\delta
	\right)
	=0.
	\]
\end{definition}

Fix $ T > 0 $. The hydrodynamic limit stated below is not the main focus of the present paper.
It can be obtained by adapting the standard entropy method, replacement lemmas and energy estimates for exclusion processes with slow boundaries, as in \cite{ref2,ref24}.  

\begin{theorem}[Hydrodynamic limit]\label{th2.1}
	Let \(\{\mu^n\}_{n\in\mathbb N}\) be a sequence of probability measures on \(\Omega_n\) associated with \(\bm\rho_0\). 
	Then, for every \(t\in[0,T]\), every \(\delta>0\), every \(k\in\{1,2\}\), and every \(f_k\in C^\infty([0,1];\mathbb R) \),
	\[
	\lim_{n\to\infty}
	\mathbb P_{\mu^n}^n
	\left(
	\left|
	X_k^{n,t}(f_k)
	-
	\int_0^1 f_k(u)\rho_k(t,u)\,du
	\right|>\delta
	\right)
	=0,
	\]
	where $ \bm\rho(t,u)
	=
	\begin{pmatrix}
		\rho_1(t,u)\\
		\rho_2(t,u)
	\end{pmatrix} $
    is the unique weak solution of the reaction-diffusion system with Robin boundary conditions
	\begin{equation}\label{eq2.1}
		\begin{cases}
			\partial_t\bm\rho(t,u)
			=
			\Delta\bm\rho(t,u)+M\bm\rho(t,u),
			& t>0,\ u\in(0,1),\\[0.2cm]
			\partial_u\bm\rho(t,0)
			=
			K_L\bm\rho(t,0)-\bm\alpha,
			& t>0,\\[0.2cm]
			\partial_u\bm\rho(t,1)
			=
			\bm\beta-K_R\bm\rho(t,1),
			& t>0,\\[0.2cm]
			\bm\rho(0,u)=\bm\rho_0(u),
			& u\in[0,1].
		\end{cases}
	\end{equation}
	Here $ \bm \alpha = \left( \alpha_1,\alpha_2 \right)^{\top} $, $ \bm \beta = \left( \beta_1,\beta_2 \right)^{\top} $, and
	\[
	M=
	\begin{pmatrix}
		-\gamma_1 & \gamma_2\\
		\gamma_1 & -\gamma_2
	\end{pmatrix},
	\qquad
	K_L=
	\begin{pmatrix}
		1-\alpha_2 & \alpha_1\\
		\alpha_2 & 1-\alpha_1
	\end{pmatrix},
	\qquad
	K_R=
	\begin{pmatrix}
		1-\beta_2 & \beta_1\\
		\beta_2 & 1-\beta_1
	\end{pmatrix}.
	\]
\end{theorem}

\subsection{  Density fluctuations }

In this subsection, we study the density fluctuations around the hydrodynamic limit. To this end, we first specify the space of test functions and some initial assumptions.

\subsubsection{Test function space and initial assumptions}

We first introduce the test function space adapted to the adjoint linearized hydrodynamic operator. 
Let
\begin{align*}
	\mathcal A^\dagger= \Delta \mathrm I_2 + M^\top.
\end{align*}
We define
\begin{equation}\label{eq2.2}
\mathscr S^\dagger
=
\left\{
\bm f\in \mathcal C^\infty([0,1];\mathbb R^2):
\begin{array}{l}
	\partial_u\left( (\mathcal A^\dagger)^m\bm f\right) (0)
	=
	K_L^\top\left( (\mathcal A^\dagger)^m\bm f \right) (0),\\[0.15cm]
	\partial_u \left( (\mathcal A^\dagger)^m\bm f\right) (1)
	=
	-
	K_R^\top \left( (\mathcal A^\dagger)^m\bm f\right) (1),
\end{array}
\quad
\forall\,m\in\mathbb N\cup\{0\}
\right\}.	 
\end{equation}
This is the natural space of smooth vector-valued test functions satisfying the adjoint Robin boundary conditions and all their compatibility conditions.

\begin{definition}
	For each $k\in\mathbb N\cup\{0\}$ and $\bm f\in \mathscr S^\dagger $, define the seminorm
	\begin{equation}\label{eq2.3}
		\|\bm f\|_k
		:=
		\sup_{u\in[0,1]}
		\bigl\|
		\partial_u^k \bm f(u)
		\bigr\| .
	\end{equation}
	We denote by $(\mathscr S^\dagger)'$ the topological dual of $\mathscr S^\dagger$ with respect to the topology generated by the seminorms. Thus \( (\mathscr S^\dagger)'\) is the space of all continuous linear functionals
	\[
	\bm Y:\mathscr S^\dagger \to\mathbb R.
	\]
\end{definition}

Fix an initial distribution \(\mu^n\) on \(\Omega_n\). For \(x\in\Sigma_n\) and \(t\ge0\), define
\[
\rho_k^{n,t}(x)
=
\mathbb E_{\mu^n}^n\left[\eta_k^x(t)\right],
\qquad
k=0,1,2.
\]
We extend these functions to the boundary by setting
\[
\rho_k^{n,t}(0)=\alpha_k,
\qquad
\rho_k^{n,t}(n)=\beta_k,
\qquad
k=0,1,2,
\]
where $ \alpha_0 = 1- \alpha_1 -\alpha_2 $ and $ \beta_0 = 1- \beta_1 -\beta_2 $. 
For $ k \in \{ 1,2 \}$, a direct computation from the generator gives 
\begin{equation}\label{eq2.4}
	\begin{cases}
		\partial_t\rho_k^{n,t}(x)
		=
		n^2
		\left[
		\rho_k^{n,t}(x+1)+\rho_k^{n,t}(x-1)-2\rho_k^{n,t}(x)
		\right]
		-\gamma_k\rho_k^{n,t}(x)
		+\gamma_{3-k}\rho_{3-k}^{n,t}(x),
		& x \in \Sigma_{n}^ \circ ,\\[0.2cm]
		\partial_t\rho_k^{n,t}(1)
		=
		n^2\left[
		\rho_k^{n,t}(2)-\rho_k^{n,t}(1)
		\right]
		+
		n\left[
		\alpha_k\rho_0^{n,t}(1)-\alpha_0\rho_k^{n,t}(1)
		\right],
		& t\ge0,\\[0.2cm]
		\partial_t\rho_k^{n,t}(n-1)
		=
		n^2\left[
		\rho_k^{n,t}(n-2)-\rho_k^{n,t}(n-1)
		\right]
		+
		n\left[
		\beta_k\rho_0^{n,t}(n-1)-\beta_0\rho_k^{n,t}(n-1)
		\right],
		& t\ge0,\\[0.2cm]
		\rho_k^{n,t}(0)=\alpha_k,
		\qquad
		\rho_k^{n,t}(n)=\beta_k,
		& t\ge0,
	\end{cases}
\end{equation}
 where $  \Sigma_{n}^ \circ:= \{ 2, \cdots, n-2 \} $.

In order to obtain the non-equilibrium fluctuations, we need the following initial assumptions.
\begin{assumption}\label{ass1}
	There exists a measurable profile
	\[
	\bm\rho_0(u)
	=
	\begin{pmatrix}
		\rho_1^0(u)\\
		\rho_2^0(u)
	\end{pmatrix},
	\qquad u\in[0,1],
	\]
	such that
	\[
	0\le \rho_1^0(u),\rho_2^0(u)\le1,
	\qquad
	\rho_1^0(u)+\rho_2^0(u)\le1.
	\]
	The sequence \( \{\mu^n\}_{n\in \mathbb{N}} \) is assumed to be associated with \(\bm \rho_0\) in the sense of Definition \ref{d2.1}.
\end{assumption}

\begin{assumption}\label{ass2}
	Let $ \bm\rho_0^{n}(x) = \begin{pmatrix}
	\rho_1^{n,0}(x)	\\
	\rho_2^{n,0}(x)
	\end{pmatrix}$. There exists a constant \(C_1>0\), independent of \(n\), such that
	\[
	\max_{x\in\Sigma_n}
	\left\|
	\bm\rho_0^{n}(x)
	-
	\bm\rho_0\left(\frac{x}{n}\right)
	\right\| 
	\le
	\frac{C_1}{n}.
	\]
\end{assumption}

\begin{assumption}\label{ass3}
	There exists a constant \(C_2>0\), independent of \(n\), such that
	\[
	\max_{1 \le x \le n-2}
	\left\|
	\bm\rho_0^{n}(x+1)
	-
    \bm\rho_0^{n}(x)
	\right\| 
	\le
	\frac{C_2}{n}.
	\]
\end{assumption}

\begin{assumption}\label{ass4}
	There exists a constant \(C_3>0\), independent of \(n\), such that for 
	\begin{align*}
	\varphi_{ij}^{n,0}(x,y)
	=
	\mathbb E_{\mu^n}^n\left[
	\eta_i^x(0)\eta_j^y(0)
	\right]
	-
	\rho_i^{n,0}(x)\rho_j^{n,0}(y),
	\qquad i,j \in \{1,2 \},	 
	\end{align*}
	it holds that
	\[
	\max_{1\le x<y\le n-1}
	\left\|
	\varphi_{ij}^{n,0}(x,y)
	\right\|
	\le
	\frac{C_3}{n}.
	\]
\end{assumption}

\subsubsection{ Non-equilibrium fluctuations }

 We now define the density fluctuation field. For a scalar test function
 \(h\in \mathcal{C}^\infty([0,1];\mathbb R)\), set
 \[
 Y_k^{n,t}(h)
 =
 \frac{1}{\sqrt n}
 \sum_{x=1}^{n-1}
 h\left(\frac{x}{n}\right)
 \left[
 \eta_k^x(t)-\rho_k^{n,t}(x)
 \right],
 \qquad
 k=1,2.
 \]
 Equivalently, for a vector-valued test function
 \[
 \bm f=(f_1,f_2)^\top\in\mathscr S^\dagger,
 \]
 we define
 \[
 \bm Y_t^n(\bm f)
 =
 \left\langle \bm Y_t^n,\bm f\right\rangle
 :=
 Y_1^{n,t}(f_1)+Y_2^{n,t}(f_2).
 \]
 That is,
 \begin{equation}\label{eq2.5}
 	\left\langle \bm Y_t^n,\bm f\right\rangle
 	=
 	\frac{1}{\sqrt n}
 	\sum_{x=1}^{n-1}
 	\left[
 	f_1\left(\frac{x}{n}\right)
 	\left(\eta_1^x(t)-\rho_1^{n,t}(x)\right)
 	+
 	f_2\left(\frac{x}{n}\right)
 	\left(\eta_2^x(t)-\rho_2^{n,t}(x)\right)
 	\right].	 
 \end{equation}
 
 For each \(n\ge1\), let \(\mathbb Q_n\) be the probability measure on $ \mathcal D([0,T],(\mathscr S^\dagger)')$ 
 given by the law of the process 
 $\{\bm Y_t^n:0\le t\le T\}$ under $\mathbb P_{\mu^n}^n$.
 
 For later use, we introduce the covariance bilinear form of the limiting noise.
 Assume that the hydrodynamic profile $ \bm \rho = \left( \rho_1, \rho_2 \right)^{\top} $ is the solution of \eqref{eq2.1}.
 For \(t\ge0\), let
 \[
 \rho_0(t,u)=1-\rho_1(t,u)-\rho_2(t,u),
 \qquad
 c(t,u)=\gamma_1\rho_1(t,u)+\gamma_2\rho_2(t,u).
 \]
 For \(\bm f=(f_1,f_2)^\top\) and \(\bm g=(g_1,g_2)^\top\), define
 \[
 Q_t(\bm f,\bm g)
 =
 Q_t^{s}(\bm f,\bm g)
 +
 Q_t^{c}(\bm f,\bm g)
 +
 Q_t^{bd}(\bm f,\bm g).
 \]
Here the conservative exchange part is
 \[
 Q_t^{s}(\bm f,\bm g)
 =
 2\int_0^1
 \partial_u\bm f(u)^\top
 \bm \chi(\bm\rho(t,u))
 \partial_u\bm g(u) \,\mathrm du,
 \]
 with
 \[
 \bm{\chi}(\bm\rho)
 =
 \begin{pmatrix}
 	\rho_1(1-\rho_1) & -\rho_1\rho_2\\
 	-\rho_1\rho_2 & \rho_2(1-\rho_2)
 \end{pmatrix}.
 \]
 The conversion part is
 \[
 Q_t^{c}(\bm f,\bm g)
 =
 \int_0^1
 c(t,u)
 \bigl(f_1(u)-f_2(u)\bigr)
 \bigl(g_1(u)-g_2(u)\bigr) \,\mathrm du.
 \]
 Finally, the boundary part is
 \[
 Q_t^{bd}(\bm f,\bm g)
 =
 \bm f(0)^\top D_L(t)\bm g(0)
 +
 \bm f(1)^\top D_R(t)\bm g(1),
 \]
 where
 \[
 D_L(t)
 =
 \begin{pmatrix}
 	\alpha_1\rho_0(t,0)+\alpha_0\rho_1(t,0) & 0\\
 	0 & \alpha_2\rho_0(t,0)+\alpha_0\rho_2(t,0)
 \end{pmatrix},
 \]
 and
 \[
 D_R(t)
 =
 \begin{pmatrix}
 	\beta_1\rho_0(t,1)+\beta_0\rho_1(t,1) & 0\\
 	0 & \beta_2\rho_0(t,1)+\beta_0\rho_2(t,1)
 \end{pmatrix}.
 \]
 
 \begin{theorem}[Non-equilibrium fluctuations]\label{th2.2}
 	Assume that the initial distributions
 	\(\{\mu^n\}_{n\in \mathbb{N}}\) satisfy Assumptions
 	\ref{ass1}--\ref{ass4}.
 	Then the sequence \(\{\mathbb Q_n\}_{n \in \mathbb{N} }\) is tight in
 	\( \mathcal D([0,T],(\mathscr S^\dagger)')\). Moreover, every limit point \(\mathbb Q\) is
 	concentrated on paths \(\{\bm Y_t:0\le t\le T\}\) satisfying the following
 	property.
 	
 	For every \(\bm f\in\mathscr S^\dagger\) and every \(t\in[0,T]\),
 	\[
 	\bm Y_t(\bm f)
 	=
 	\bm Y_0(S_t\bm f)
 	+
 	\mathcal W_t(\bm f),
 	\]
 	where \((S_t)_{t\ge0}\) is the semigroup generated by $ \mathcal A^\dagger =\Delta \mathrm I_2+M^\top $
    on \(\mathscr S^\dagger \), and $ \mathcal W_t(\bm f) $ is a mean zero Gaussian variable of variance
    \[
    \int_0^t
    Q_r\left(
    S_{t-r}\bm f,
    S_{t-r}\bm f
    \right)\,dr.
    \]
 	Moreover, $ \bm Y_0 $ and $ \mathcal W_t $ are uncorrelated in the sense that $ \mathbb{E}  \left[ \bm Y_0(\bm f)\mathcal W_t(\bm g)  \right] = 0 $ for all $ \bm f,\bm g \in \mathscr{S}^{\dagger } $.
 \end{theorem}

 \begin{theorem}[Ornstein-Uhlenbeck limit]\label{th2.3}
 	Assume that the sequence of initial fluctuation fields $ \left\{\bm Y_0^n\right\}_{n\in \mathbb{N}} $
 	converges, as \(n\to\infty\), to a mean-zero Gaussian field \(\bm Y_0\) on
 	\( (\mathscr S^\dagger)'\), whose covariance is given by
 	\[
 	\lim_{n\to\infty}
 	\mathbb E_{\mu^n}^n
 	\left[
 	\bm Y_0^n(\bm f)\bm Y_0^n(\bm g)
 	\right]
 	=
 	\mathbb E
 	\left[
 	\bm Y_0(\bm f)\bm Y_0(\bm g)
 	\right]
 	=:
 	\sigma(\bm f,\bm g),
 	\qquad
 	\bm f,\bm g\in\mathscr S^\dagger.
 	\]
 	Then the sequence \(\{\mathbb Q_n\}_{n \in \mathbb{N} }\) converges in
 	\(\mathcal D([0,T],(\mathscr S^\dagger)')\) to the generalized Ornstein-Uhlenbeck process \(\{\bm Y_t:0\le t\le T\}\) characterized by the
 	martingale problem
 	\[
 	\mathcal M_t(\bm f)
 	=
 	\bm Y_t(\bm f)-\bm Y_0(\bm f)
 	-
 	\int_0^t
 	\bm Y_s(\mathcal A^\dagger \bm f) \,\mathrm ds,
 	\qquad
 	\bm f\in\mathscr S^\dagger,
 	\]
 	where \(\mathcal M_t(\bm f)\) is a continuous centered Gaussian martingale with
 	quadratic covariation
 	\[
 	\left\langle
 	\mathcal M(\bm f),\mathcal M(\bm g)
 	\right\rangle_t
 	=
 	\int_0^t
 	Q_s(\bm f,\bm g) \,\mathrm ds.
 	\] 	
 	Equivalently, for every \(0\le s\le t\le T\) and every
 	\(\bm f,\bm g\in\mathscr S^\dagger\),
 	\[
 	\mathbb E\left[
 	\bm Y_t(\bm f)\bm Y_s(\bm g)
 	\right]
 	=
 	\sigma(S_t\bm f,S_s\bm g)
 	+
 	\int_0^s
 	Q_r\left(
 	S_{t-r}\bm f,
 	S_{s-r}\bm g
 	\right) \,\mathrm dr.
 	\]
 \end{theorem}

\begin{remark}
The limiting process can be formally interpreted as the solution of
\begin{align*}
\mathrm d\bm Y_t
&=
(\Delta \mathrm I_2+M)\bm Y_t\,\mathrm dt
+
\nabla\cdot\left(\sqrt{2\bm \chi(\bm\rho_t)}\,\mathrm d\bm W_t^{s}\right)
+
\sqrt{c(t,\cdot)}
\begin{pmatrix}
	-1\\
	1
\end{pmatrix}
\mathrm d W_t^{c} \\
&\quad +
\sqrt{D_L(t)}\,\mathrm d\bm B_t^L\,\delta_0
+
\sqrt{D_R(t)}\,\mathrm d\bm B_t^R\,\delta_1 .	 
\end{align*}
Here \(\bm W^s\) is a two-dimensional space-time white noise, \(W^c\) is a
scalar space-time white noise, and \(\bm B^L,\bm B^R\) are two-dimensional
Brownian motions. All these noises are mutually independent. 
This expression is purely formal. Its rigorous interpretation is given by the
martingale problem above, whose quadratic variation is determined by the
bilinear form \(Q_t=Q_t^s+Q_t^c+Q_t^{bd}\).	 
\end{remark}

\section{Semigroup results}

In this section we collect the semigroup results needed in the sequel. We consider both the homogeneous linear equation associated with the hydrodynamic equation and its adjoint equation. The semigroup associated with the original operator will be denoted by \((T_t)_{t\ge0}\), while the semigroup associated with the adjoint operator will be denoted by \((S_t)_{t\ge0}\).
Let
\[
\mathcal A=\Delta \mathrm I_2+M,
\qquad
\mathcal A^\dagger=\Delta \mathrm I_2+M^\top .
\]
Recalling the space $\mathscr S^\dagger $ introduced in \eqref{eq2.2}, we analogously define 
\begin{equation}\label{eq3.1}
	\mathscr S
	=
	\left\{
	\bm f\in \mathcal{C}^\infty\left([0,1];\mathbb R^2\right):
	\begin{array}{l}
		\partial_u(\mathcal A^m\bm f)(0)
		=
		K_L(\mathcal A^m\bm f)(0),\\[0.15cm]
		\partial_u(\mathcal A^m\bm f)(1)
		=
		-K_R(\mathcal A^m\bm f)(1),
	\end{array}
	\quad \forall\,m\in\mathbb N\cup\{0\}
	\right\}.
\end{equation}

    We consider the original homogeneous linear problem
    \begin{align}\label{eq3.2}
    	\begin{cases}
    		\partial_t \bm{\rho}(t,u) = \Delta \bm {\rho}(t,u) + M  \bm {\rho}(t,u), & \text{for } t>0,u\in (0,1), \\
    		\partial_u \bm {\rho}(t,0) = K_L  \bm {\rho}(t,0), & \text{for } t>0, \\ 
    		\partial_u \bm {\rho}(t,1) = - K_R \bm {\rho}(t,1) , & \text{for } t>0, \\
    		\bm {\rho}(0,u) = \bm {\rho}_0(u), & u \in [0,1],
    	\end{cases}
    \end{align}	
    and the adjoint homogeneous linear problem
    \begin{align}\label{eq3.3}
    	\begin{cases}
    		\partial_t \bm{\rho}(t,u) = \Delta \bm {\rho}(t,u) + M^{\top} \bm {\rho}(t,u), & \text{for } t>0,u\in (0,1), \\
    		\partial_u \bm {\rho}(t,0) = K_L^{\top} \bm {\rho}(t,0), & \text{for } t>0, \\ 
    		\partial_u \bm {\rho}(t,1) = - K_R^{\top} \bm {\rho}(t,1) , & \text{for } t>0, \\
    		\bm {\rho}(0,u) = \bm {\rho}_0(u), & u \in [0,1].
    	\end{cases}
    \end{align}
    Let  
    \[
    D(\mathcal{A} )
    =
    \left\{
    \bm f\in \mathcal{H}^{2 }\left((0,1);\mathbb R^2\right):~
    	\partial_u \bm f(0)
    	=
    	K_L \bm f(0), 
    	\partial_u \bm f(1)
    	=
    	-K_R \bm f(1)
    \right\},
    \]	 
    \[
    D(  \mathcal{A}^{\dagger} ) = 
    \left\{
    \bm f\in \mathcal H^{2}((0,1);\mathbb R^2):~
    	\partial_u \bm f  (0)
    	=
    	K_L^\top \bm f   (0), 
    	\partial_u \bm f  (1)
    	=
    	-
    	K_R^\top  \bm f (1)
    \right\},
    \]	 
    \[
    D(\mathcal{A}^m ) = \left\{
    \bm f\in D(\mathcal{A}^{m-1} ):~ \mathcal{A}^{m-1} \bm f \in D(\mathcal{A} )
    \right\},
    \]
    \[
    D((\mathcal{A}^\dagger)^m ) = \left\{
    \bm f\in D((\mathcal{A}^\dagger)^{m-1} ):~ (\mathcal{A}^\dagger)^{m-1} \bm f \in D(\mathcal{A}^\dagger )
    \right\}
    \]
    for $ m \ge 2 $.
\begin{proposition}\label{p3.1}
	The operators \(\mathcal A\) and $ \mathcal A^\dagger $ generate analytic \(C_0\)-semigroups
	\((T_t)_{t\ge0}\) and \((S_t)_{t\ge0}\) on \( \mathcal L^2((0,1);\mathbb R^2)\).
	Consequently, for every initial datum in \(\mathcal L^2((0,1);\mathbb R^2)\), equation \eqref{eq3.2} $\left( \eqref{eq3.3}\right)$  admits a unique mild solution given by $ \bm\rho(t)=T_t\bm\rho_0 $ $ \left( \bm \rho(t)=S_t\bm \rho_0 \right) $.
	Moreover, these solutions are smooth in space and time for every \(t>0\).
	If \( \bm\rho_0\in\mathscr S \) $ \left( \bm \rho_0\in\mathscr S^\dagger \right) $, then $ T_t\bm\rho_0 $ $ \left( S_t\bm \rho_0 \right)  $ will be
	$\mathcal C^\infty$ in space and time for all $t\ge0$.
\end{proposition}
\begin{proof}
	 We present the proof for the adjoint operator $ \mathcal A^\dagger $ only; the proof for $ \mathcal A $ is analogous.
	 Let
	 \[
	 H= \mathcal L^2((0,1);\mathbb R^2),
	 \qquad
	 V= \mathcal H^1((0,1);\mathbb R^2),
	 \]
	 with norms
	 \[
	 \|\bm f\|_H^2
	 =
	 \int_0^1 |\bm f(x)|^2\,\mathrm dx,
	 \qquad
	 \|\bm f\|_V^2
	 =
	 \|\bm f\|_H^2+\|\partial_x\bm f\|_H^2.
	 \]
	 Let \(j:V\to H\) be the canonical embedding \(j(\bm u)=\bm u\). Then \(j\) is a
	 bounded linear operator and \(j(V)\) is dense in \(H\).
	 
	 Define the bilinear form \(a^\dagger:V\times V\to\mathbb R\) by
	 \begin{equation}\label{eq3.4}
	 	a^\dagger(\bm u,\bm v)
	 	=
	 	\int_0^1
	 	\partial_x\bm u(x)\cdot \partial_x\bm v(x)\,\mathrm dx
	 	-
	 	\int_0^1
	 	M^\top\bm u(x)\cdot \bm v(x)\,\mathrm dx
	 	+
	 	K_L^\top\bm u(0)\cdot \bm v(0)
	 	+
	 	K_R^\top\bm u(1)\cdot \bm v(1).
	 \end{equation}
	 We first verify that the bilinear form $ a^\dagger $ is continuous on $ V $, namely, there exists a constant $C>0$  such that
	 \begin{align*}
	 		|a^\dagger (\bm u,\bm v)| \le C\|\bm u\|_V\|\bm v\|_V 
	 \end{align*}
	 for all $ \bm u,\bm v\in V $.
	 By the one-dimensional trace theorem, there exists a constant \(C_{\mathrm{tr}}>0\) such that
	 \[
	 \|\bm u(0)\|+\|\bm u(1)\|
	 \le
	 C_{\mathrm{tr}}\|\bm u\|_V,
	 \qquad
	 \bm u\in V.
	 \]
	 Therefore, for all $ \bm u,\bm v\in V $,
	 \[
	 \left|
	 K_L^\top\bm u(0)\cdot \bm v(0)
	 \right|
	 +
	 \left|
	 K_R^\top\bm u(1)\cdot \bm v(1)
	 \right|
	 \le
	 C\|\bm u\|_V\|\bm v\|_V.
	 \]
	 Since \(M\) is a bounded matrix, the Cauchy-Schwarz inequality gives
	 \[
	 |a^\dagger(\bm u,\bm v)|
	 \le
	 C\|\bm u\|_V\|\bm v\|_V,
	 \qquad
	 \bm u,\bm v\in V.
	 \]
	 Thus \(a^\dagger\) is continuous.
	 
	 We next verify that $ a^\dagger $ is $ j $-elliptic; that is, there exist $ \omega\in\mathbb R $ and $ \mu>0 $ such that
	 \begin{align*}
	 	 a(\bm u,\bm u)+\omega\|j(\bm u)\|_H^2
	 	\ge
	 	\mu\|\bm u\|_V^2 
	 \end{align*}
	 for all $ \bm u \in V $.
	 For \( \bm u\in V \),
	 \[
	 a^\dagger(\bm u,\bm u)
	 =
	 \int_0^1 |\partial_x \bm u(x) | ^2 \,\mathrm dx
	 -
	 \int_0^1 M^\top\bm u(x)\cdot \bm u(x)\,\mathrm dx
	 +
	 K_L^\top\bm u(0)\cdot \bm u(0)
	 +
	 K_R^\top\bm u(1)\cdot \bm u(1).
	 \]
	 By Young's inequality, we have
	 \begin{align*}
	 	\int_0^1 M^{\top}\bm u(x)\cdot \bm u(x)\,\mathrm dx
	 	&= -\int_0^1 \left( \gamma_1 u_1^2(x) + \gamma_2 u_2^2(x) - \left( \gamma_1+ \gamma_2 \right) u_1(x) u_2(x) \right) \,\mathrm dx \\
	 	&\le -\int_0^1 \left( \gamma_1 u_1^2(x) + \gamma_2 u_2^2(x) \right) \,\mathrm dx
	 	+ \frac{ \gamma_1 + \gamma_2 }{2} \int_0^1 \left( u_1^2(x) + u_2^2(x) \right) \,\mathrm dx \\
	 	&= \frac{1}{2} \int_0^1 \left( \left( \gamma_2 - \gamma_1 \right) u_1^2(x)
	 	+ \left( \gamma_1 - \gamma_2 \right) u_2^2(x)  \right) \,\mathrm dx \\
	 	&\le \frac{|\gamma_1 - \gamma_2|}{2} \| \bm{u} \|_{H}^2.
	 \end{align*}
	 Moreover, for every \(\varepsilon>0\), the trace inequality with parameter gives
	 \[
	 |\bm u(0)|^2+|\bm u(1)|^2
	 \le
	 \varepsilon\|\partial_x\bm u\|_H^2
	 +
	 C_\varepsilon\|\bm u\|_H^2.
	 \]
	 Consequently,
	 \[
	 K_L^\top\bm u(0)\cdot \bm u(0)
	 +
	 K_R^\top\bm u(1)\cdot \bm u(1)
	 \ge
	 -
	 C\varepsilon\|\partial_x\bm u\|_H^2
	 -
	 C_\varepsilon\|\bm u\|_H^2.
	 \]
	 Then
	 \begin{align*}
	 	a^\dagger (\bm u,\bm u) \ge 
	 	\left( 1 - C\varepsilon \right) \|\partial_x \bm u\|_H^2 - \left( \frac{|\gamma_1 - \gamma_2|}{2} + C_\varepsilon \right)  \| \bm{u} \|_{H}^2.
	 \end{align*}
	 We now fix $ \varepsilon  $  small enough so that $ 1 - C\varepsilon \ge \frac{1}{2} $.
	 Set $ \omega = \frac{1 + | \gamma_1 - \gamma_2|}{2} + C_\varepsilon  $, then
	 \begin{align*}
	 	a^\dagger (\bm u,\bm u)+\omega\|j(\bm u)\|_H^2
	 	\ge
	 	\frac{1}{2}\|\partial_x \bm u\|_H^2 + \frac{1}{2} \|\bm u\|_H^2
	 	= \frac{1}{2} \|\bm u\|_V^2.
	 \end{align*}
	 Thus \(a^\dagger\) is \(j\)-elliptic.
	 
	 By Theorem~2.1 in \cite{ref1}, the operator associated with \(a^\dagger\) is \(m\)-sectorial. We now identify this operator.
     Let $L_{\mathrm{Rob}}$ be the operator defined by
	\[
	L_{\mathrm{Rob}}\bm u
	=
	- \Delta \bm u-M^\top \bm u = - \mathcal{A}^\dagger \bm u
	\]
	on the domain
	\[
	D(L_{\mathrm{Rob}})
	=
	\left\{
	\bm f\in \mathcal  H^2((0,1);\mathbb R^2):
	\partial_u\bm f(0)=K_L^\top \bm f(0),
	\partial_u\bm f(1)=-K_R^\top \bm f(1)
	\right\} = D(  \mathcal{A}^{\dagger} ).
	\]
	For every \(\bm u\in D(L_{\mathrm{Rob}} )\) and every \(\bm v\in V\), integration by
	parts gives
	\[
	a^\dagger (\bm u,\bm v)=(L_{\mathrm{Rob}}\bm u,\bm v)_H.
	\]
	Conversely, suppose that \(\bm u\in V\) and \(\bm h\in H\) satisfy
	\[
	a^\dagger(\bm u,\bm v)
	=
	(\bm h,\bm v)_H,
	\qquad
	\forall\,\bm v\in V.
	\]
	Taking first \(\bm v\in \mathcal C_c^\infty((0,1);\mathbb R^2)\), we obtain
	\[
	-\Delta\bm u-M^\top\bm u=\bm h
	\]
	in the weak sense on \((0,1)\).
	Since \(\bm h\in H\), elliptic regularity implies $ \bm u\in \mathcal H^2((0,1);\mathbb R^2) $.
	Then, integrating by parts for arbitrary $\bm v\in V$, we obtain
	\[
	\bigl(-\partial_x\bm u(0)+K_L^\top \bm u(0)\bigr)\cdot \bm v(0)
	+
	\bigl(\partial_x\bm u(1)+K_R^\top \bm u(1)\bigr)\cdot \bm v(1)
	=0.
	\]
	Since the trace map
	\[
	V=H^1((0,1);\mathbb R^2)\to \mathbb R^2\times\mathbb R^2,
	\qquad
	\bm v\mapsto (\bm v(0),\bm v(1))
	\]
	is surjective, it follows that
	\[
	\partial_x\bm u(0)=K_L^\top \bm u(0),
	\qquad
	\partial_x\bm u(1)=-K_R^\top \bm u(1).
	\]
	Therefore $\bm u\in D(L_{\mathrm{Rob}})$ and $ L_{\mathrm{Rob}}\bm u= \bm h $.
	Hence the operator associated with \(a^\dagger\) is \(L_{\mathrm{Rob}} = - \mathcal A^\dagger\).
	It follows that \(\mathcal A^\dagger\) generates an analytic \(C_0\)-semigroup
	\((S_t)_{t\ge0}\) on \(H\). Hence the mild solution of
	\eqref{eq3.3} is \(S_t\bm \rho_0\).
	
	It remains to justify the smoothing properties. Since \(\mathcal A^\dagger\)
	generates an analytic \(C_0\)-semigroup, we have that, for every \(\bm \rho_0\in \mathcal L^2((0,1);\mathbb R^2) \), every \(t>0\), and every \(m\ge1\),
    \[
    S_t\bm \rho_0\in D((\mathcal A^\dagger)^m).
    \]
    Moreover,
    \[
    \partial_t^k S_t\bm \rho_0
    =
    (\mathcal A^\dagger)^k S_t\bm \rho_0,
    \qquad
    t>0,\ k\ge1.
    \]
    By the elliptic regularity for the Robin problem,
    \[
    D((\mathcal A^\dagger)^m)
    \subset
    H^{2m}((0,1);\mathbb R^2),
    \qquad
    m\ge1.
    \]
    Therefore \(S_t\bm \rho_0\) is smooth in space and time for every \(t>0\).
 
	If \(\bm \rho_0\in\mathscr S^\dagger\), by definition, we have $ \bm \rho_0\in D((\mathcal A^\dagger)^m) $ for every $ m\ge0 $.
	Since the analytic semigroup preserves the domains of powers of its generator
	and commutes with the generator on these domains, we obtain
	\[
	S_t\bm \rho_0\in D((\mathcal A^\dagger)^m),
	\qquad
	(\mathcal A^\dagger)^mS_t\bm \rho_0
	=
	S_t(\mathcal A^\dagger)^m\bm \rho_0,
	\qquad
	t\ge0,\ m\ge0.
	\]
	Hence \(S_t\bm \rho_0\in\mathscr S^\dagger\) for every \(t\ge0\). The proof for the original
	operator \(\mathcal A\) is analogous and gives the corresponding statements for
	\((T_t)_{t\ge0}\) and \(\mathscr S \).

\end{proof}

\begin{corollary}\label{c3.1}
	For every \(\bm f\in \mathcal L^2((0,1);\mathbb R^2)\) and every \(t>0\), one has
	\[
	T_t\bm f\in\mathscr S,
	\qquad
	\mathcal A^mT_t\bm f\in\mathscr S,
	\qquad
	m\in\mathbb N.
	\]
	Similarly,
	\[
	S_t\bm f\in\mathscr S^\dagger,
	\qquad
	(\mathcal A^\dagger)^mS_t\bm f\in\mathscr S^\dagger,
	\qquad
	m\in\mathbb N.
	\]
\end{corollary}

\begin{proof}
	
	We prove the statement for the adjoint semigroup. The original semigroup is treated in
	the same way.
	
	Let \(\bm f\in \mathcal L^2((0,1);\mathbb R^2)\) and \(t>0\). Based on the above proof, $ S_t\bm f\in D((\mathcal A^\dagger)^m) $ for every $ m\in\mathbb N  $.
	In particular, for every \(m\in\mathbb N\cup\{0\}\),
	\[
	(\mathcal A^\dagger)^mS_t\bm f\in D(\mathcal A^\dagger).
	\]
	By the definition of \(D(\mathcal A^\dagger)\), this implies
	\[
	\partial_u\bigl((\mathcal A^\dagger)^mS_t\bm f\bigr)(0)
	=
	K_L^\top
	\bigl((\mathcal A^\dagger)^mS_t\bm f\bigr)(0), \qquad 
	\partial_u\bigl((\mathcal A^\dagger)^mS_t\bm f\bigr)(1)
	=
	-
	K_R^\top
	\bigl((\mathcal A^\dagger)^mS_t\bm f\bigr)(1).
	\]
	Therefore \(S_t\bm f\in\mathscr S^\dagger \) and $ (\mathcal A^\dagger)^mS_t\bm f\in\mathscr S^\dagger $ for every \(m\in\mathbb N\cup\{0\}\). The proof for \(T_t\) and \(\mathscr S \) is identical.
	
\end{proof}

	We next prove the exponential decay of the semigroups.
	
	\begin{lemma}\label{l3.1}
		There exist constants \(C>0\) and \(\kappa>0\) such that, for every
		\(\bm f\in \mathcal L^2((0,1);\mathbb R^2)\),
		\[
		\|T_t\bm f\|_{\mathcal L^2}
		\le
		Ce^{-\kappa t}\|\bm f\|_{\mathcal L^2},
		\qquad t\ge0.
		\]
		Consequently, the adjoint semigroup also satisfies
		\[
		\|S_t\bm f\|_{\mathcal  L^2}
		\le
		Ce^{-\kappa t}\|\bm f\|_{\mathcal L^2},
		\qquad t\ge0.
		\]
		In particular,
		\[
		\lim_{t\to\infty}T_t\bm f=0,
		\qquad
		\lim_{t\to\infty}S_t\bm f=0
		\]
		in \(\mathcal  L^2((0,1);\mathbb R^2)\).
\end{lemma}

\begin{proof}
	We first prove the estimate for the original semigroup \((T_t)_{t\ge0}\).
	Let $ \bm \rho = (\rho_1,\rho_2)^\top =  T_t\bm \rho_0(u)$ solve equation \eqref{eq3.2}.
	Define
	\[
	U=\rho_1+\rho_2,
	\qquad
	V=\gamma_1 \rho_1-\gamma_2 \rho_2,
	\qquad
	\gamma=\gamma_1+\gamma_2.
	\]
	A direct computation shows that \(U\) solves the scalar heat equation
	\begin{equation}\label{eq}
		\begin{cases}
			\partial_t U (t,u) =\partial_u^2U (t,u),
			&  t > 0 ,u\in(0,1),\\
			\partial_u U(t,0)=U(t,0), 	&  t > 0 \\
			\partial_u U(t,1)=-U(t,1),	&  t > 0 \\
			U(0,u)= U_0(u), & u \in [0,1].
		\end{cases}
	\end{equation}
	By Proposition 3.1 and Corollary 3.3 in \cite{ref3}, this equation admits the explicit representation 
	\begin{align*}
		U(t,u)=  \sum_{n=1}^{\infty} a_n e^{-\lambda_n t}\Psi_n(u),
	\end{align*}
	where $\{\Psi_n\}_{n\in\mathbb{N}}$ is an orthonormal basis of $\mathcal L^2\left( [0,1]; \mathbb{R} \right) $ constituted by eigenfunctions of the associated regular Sturm-Liouville problem, $a_n$ are the Fourier coefficients of $U_0$ in that basis, and $ \lambda_n  \sim  n^2 \pi^2 $. Moreover, for every initial datum $ U_0(u) \in \mathcal L^2\left( [0,1]; \mathbb{R} \right) $, the solution decays to zero exponentially fast as $ t \to \infty $.
	 
	We now consider \(V\). By direct computation, \(V\) satisfies
	\begin{equation}\label{eq3.6}
		\begin{cases}
			\partial_t V (t,u) =\partial_u^2V (t,u) -\gamma V,
			& t > 0 ,u\in(0,1),\\
			\partial_u V(t,0)= \alpha_0 V(t,0) + q_0U(t,0), 	& t > 0, \\
			\partial_u V(t,1)= -\beta_0 V(t,1) - q_1U(t,1),	& t > 0, \\
			V(0,u)= V_0(u), & u \in [0,1],
		\end{cases}
	\end{equation}
	where
	\[
	 \alpha_0 =1-\alpha_1-\alpha_2 > 0,
	\qquad
	 \beta_0 =1-\beta_1-\beta_2 > 0,
	 \]
	 \[
	 q_0=\gamma_1\alpha_1-\gamma_2\alpha_2,
	 \qquad
	 q_1=\gamma_1\beta_1-\gamma_2\beta_2.
	\]
    Multiplying \eqref{eq3.6} by \(V\) and integrating over \((0,1)\), we obtain
	\begin{align*}
		\frac12\frac{\mathrm d}{\mathrm dt}\|V(t)\|_{\mathcal L^2}^2
		&=
		-\|\partial_uV(t)\|_{\mathcal L^2}^2
		-\gamma\|V(t)\|_{\mathcal L^2}^2
		+
		V(t,1)\partial_uV(t,1)
		-
		V(t,0)\partial_uV(t,0) \\
		&= - \|\partial_uV(t)\|_{\mathcal L^2}^2
		- \gamma\|V(t)\|_{L^2}^2 
		- \alpha_0 V(t,0)^2 - \beta_0 V(t,1)^2 \\
		&\quad - q_0U(t,0)V(t,0) - q_1U(t,1)V(t,1).
	\end{align*}
	By Young's inequality, we have
	\begin{align*}
		&| q_0U(t,0)V(t,0) | \le \frac{\alpha_0}{2} V(t,0)^2 + \frac{q_0^2}{2\alpha_0} U(t,0)^2,\\
		&| q_1U(t,1)V(t,1) | \le \frac{\beta_0}{2} V(t,1)^2 + \frac{q_1^2}{2\beta_0} U(t,1)^2.
	\end{align*}
	Therefore,
	\begin{equation}\label{eq3.7}
		\frac{\mathrm d}{\mathrm dt}\|V(t)\|_{\mathcal L^2}^2
		\le
		-2 \gamma\|V(t)\|_{\mathcal L^2}^2
		+
		C\left(U(t,0)^2+U(t,1)^2\right).
	\end{equation}
	By the one-dimensional trace inequality,
	\[
	U(t,0)^2+U(t,1)^2
	\le
	C\|U(t)\|_{L^2}\|U(t)\|_{\mathcal H^1}
	\le
	C(1+t^{-1/2})e^{-2\lambda_1 t}\|U_0\|_{\mathcal L^2}^2,
	\]
	where $ \lambda_1 > 0 $.
	Substituting this into \eqref{eq3.7} and applying Gronwall's inequality, we obtain
	\begin{align*}
	\|V(t)\|_{\mathcal L^2}^2
	&\le
	e^{-2\gamma t}\|V_0\|_{\mathcal L^2}^2
	+
	C\int_0^t e^{-2 \gamma (t-s)}(1+s^{-1/2})e^{-2\lambda_1 s}\,\mathrm ds\
	\|U_0\|_{\mathcal L^2}^2 .
	\end{align*}
	Let \(0<\kappa<\min\{\gamma,\lambda_1 \}\). We have
	\[
	\int_0^t
	e^{-2\gamma(t-s)}
	(1+s^{-1/2})e^{-2\lambda_1 s}\,\mathrm ds
	\le
	C_\kappa e^{-2\kappa t},
	\]
	because \(s^{-1/2}\) is integrable near \(0\). Hence
	\[
	\|V(t)\|_{\mathcal L^2}^2
	\le
	C_\kappa e^{-2\kappa t}
	\left(
	\|V_0\|_{\mathcal L^2}^2+\|U_0\|_{\mathcal L^2}^2
	\right).
	\]

	Finally, since
	\[
	\rho_1=\frac{V+\gamma_2U}{\gamma_1+\gamma_2},
	\qquad
	\rho_2=\frac{\gamma_1U-V}{\gamma_1+\gamma_2},
	\]
	we conclude that
	\[
	\|\bm \rho(t)\|_{\mathcal L^2}
	\le
	C\left(
	\|U(t)\|_{L^2}+\|V(t)\|_{\mathcal L^2}
	\right)
	\le
	Ce^{-\kappa t}\|\bm \rho_0 \|_{\mathcal L^2}.
	\]
This proves the exponential decay of \(T_t\).

It remains to pass to the adjoint semigroup. Since \(\mathcal A^\dagger\) is the
\(\mathcal L^2\)-adjoint of \(\mathcal A\), the semigroup \(S_t\) is the \(\mathcal L^2\)-adjoint of
\(T_t\). Therefore, for every \(\bm f\in \mathcal L^2((0,1);\mathbb R^2)\),
\[
\|S_t\bm f\|_{\mathcal L^2}
=
\sup_{\|\bm g\|_{\mathcal L^2}=1}
\left|
\langle S_t\bm f,\bm g\rangle_{\mathcal L^2}
\right|
=
\sup_{\|\bm g\|_{\mathcal L^2}=1}
\left|
\langle \bm f,T_t\bm g\rangle_{\mathcal L^2}
\right|.
\]
Using the estimate already proved for \(T_t\), we obtain
\[
\|S_t\bm f\|_{\mathcal L^2}
\le
Ce^{-\kappa t}\|\bm f\|_{\mathcal L^2}.
\]
The proof is complete.
\end{proof}

\section{Replacement lemma}

In this section, we prove a replacement lemma which will play an important role in the
identification of the limiting quadratic variation. We first introduce a reference
product measure.
Let $ \bm r=(r_0,r_1,r_2):[0,1]\to[0,1]^3 $
be a Lipschitz continuous profile such that
\[
r_0(u)+r_1(u)+r_2(u)=1,
\qquad u\in[0,1],
\]
and assume that there exist constants \(0<a<b<1\) such that
\[
a\le r_i(u)\le b,
\qquad
i=0,1,2,\quad u\in[0,1].
\]
We also assume that \(\bm r\) is locally constant near the boundary and satisfies $ \bm r(0)= \left( \alpha_0, \alpha_1,\alpha_2  \right)  $ and $ \bm r(1)= \left( \beta_0, \beta_1, \beta_2 \right)  $.
Define the product measure \(\nu_{\bm r}^n\) on \(\Omega_n\) by
\[
\nu_{\bm r}^n(\eta)
=
\prod_{x=1}^{n-1}
\prod_{i=0}^2
r_i\left(\frac{x}{n}\right)^{\bm{1}_{\{ \eta_i(x)= 1\}}}.
\]

We first give the following Dirichlet form estimates.

\begin{lemma}\label{l4.1}
	There exist constants \( C,M>0 \) and \( N_0 \in \mathbb{N} \), such that, for every
	density \(f\) with respect to \(\nu_{\bm r}^n\),
	\[
	\left\langle L_n^\pm \sqrt f,\sqrt f\right\rangle_{\nu_{\bm r}^n}
	=
	-\frac12 D_{\nu_{\bm r}^n}^\pm(\sqrt f), \quad \forall  n \ge N_0,
	\]
	 \[
	 \left\langle L_n^s \sqrt f,\sqrt f\right\rangle_{\nu_{\bm r}^n}
	 \le
	 -MD_{\nu_{\bm r}^n}^s(\sqrt f)+\frac{C}{n}, \quad \forall n \ge 1,
	 \]
	 and
	 \[
	 \left\langle L_n^c \sqrt f,\sqrt f\right\rangle_{\nu_{\bm r}^n}
	 \le
	 Cn, \quad \forall n \ge 1.
	 \]
	 Here 
	 \begin{align*}
	 	 D_{\nu_{\bm r}^n}^-(\sqrt{f}) 
	 	 &= \frac{1}{n} \int \sum_{k=1}^{2} \left[ \alpha_k \eta_0^1 \left(  \sqrt{f}(\eta - \delta_0^1 + \delta_{k}^1 ) - \sqrt{f}(\eta) \right) ^2 \right. \\
	 	 &\quad \left. 
	 	 + \alpha_0 \eta_k^1 \left(  \sqrt{f}(\eta - \delta_{k}^1 + \delta_0^1 ) - \sqrt{f}(\eta) \right) ^2 \right] \,\mathrm d \nu_{\bm r}^n,
	 \end{align*}
	 \begin{align*}
	 	D_{\nu_{\bm r}^n}^+(\sqrt{f}) 
	 	&= \frac{1}{n} \int \sum_{k=1}^{2} \left[ \beta_k \eta_0^{n-1} \left(  \sqrt{f}(\eta - \delta_0^{n-1} + \delta_{k}^{n-1} ) - \sqrt{f}(\eta) \right) ^2 \right. \\
	 	&\quad \left. + \beta_0 \eta_k^{n-1} \left(  \sqrt{f}(\eta - \delta_{k}^{n-1} + \delta_0^{n-1} ) - \sqrt{f}(\eta) \right) ^2 \right] \,\mathrm d \nu_{\bm r}^n,
	 \end{align*}
	 and
	 \begin{align*}
	 	D_{\nu_{\bm r}^n}^s(\sqrt{f}) 
	 	= \int \sum_{x=1}^{n-2} \sum_{k,l=0}^{2} \eta_k^x \eta_l^{x+1}  
	 	\left(  \sqrt{f}(\eta -\delta_{k}^{x} + \delta_{l}^{x} + \delta_{k}^{x+1} - \delta_{l}^{x+1}) - \sqrt{f}(\eta) \right)^2 \,\mathrm d \nu_{\bm r}^n.
	 \end{align*}
\end{lemma}

\begin{proof}
	
	We first prove the boundary identity. We only treat the left boundary; the right
	boundary is identical. Since \(\bm r\) is locally constant near the left boundary and
	\(\bm r(0)=\left( \alpha_0, \alpha_1,\alpha_2  \right)\), we have, for \(n\) sufficiently
	large,
	\[
	r_i\left(\frac1n\right)=\alpha_i,
	\qquad i=0,1,2.
	\]
	Consequently, the single-site marginal at \(x=1\) satisfies the detailed balance relation 
	\[ 
	\alpha_k r_0\left(\frac1n\right)
	=
	\alpha_0 r_k\left(\frac1n\right),
	\qquad k=1,2.
	\]
	 We have
\begin{align*}
	D_{\nu_{\bm r}^n}^-(\sqrt{f})
	&= \frac1n \int \sum_{k=1}^{2}  \left[  \alpha_k \eta_0^1 \sqrt{f}(\eta)\bigl(\sqrt{f} (\eta)-2\sqrt{f}(\eta-\delta_0^1 +\delta_k^1)\bigr) \right.\\
	&\left. \quad + \alpha_0 \eta_k^1 \sqrt{f}(\eta)\bigl(\sqrt{f}(\eta)-2\sqrt{f}(\eta-\delta_k^1 +\delta_0^1)\bigr) \right] \, \mathrm d\nu_{\bm r}^n \\
	&\quad + \frac1n \int \sum_{k=1}^{2} \Bigl[ \alpha_k \eta_0^1 f(\eta-\delta_0^1 +\delta_k^1) 
	+ \alpha_0 \eta_k^1 f(\eta-\delta_k^1 +  \delta_0^1 ) \Bigr]\, \mathrm d\nu_{\bm r}^n.
\end{align*}
Performing the changes of variables $ \eta = \xi + \delta_0^1 - \delta_k^1 $ and $ \eta = \zeta -\delta_0^1 + \delta_k^1 $, respectively, we obtain
\begin{align*}
	&\int \sum_{k=1}^{2} \Bigl[ \alpha_k \eta_0^1 f(\eta-\delta_0^1 +\delta_k^1) 
	+ \alpha_0 \eta_k^1 f(\eta-\delta_k^1 +  \delta_0^1 ) \Bigr]\, \mathrm d\nu_{\bm r}^n \\
	&  = \int \sum_{k=1}^{2} \alpha_k \xi_k^1 f(\xi)  \frac{r_0(1/n)}{r_k(1/n)}\, \mathrm d\nu_{\bm r}^n
	+ \int \sum_{k=1}^{2} \alpha_0 \zeta_0^1 f(\zeta)  \frac{r_k(1/n)}{r_0(1/n)}\, \mathrm d\nu_{\bm r}^n \\
	&  = \int \sum_{k=1}^{2} \bigl[ \alpha_0 \eta_k^1 + \alpha_k \eta_0^1 \bigr] f(\eta)\, \mathrm d\nu_{\bm r}^n.
\end{align*}
Consequently,
\begin{align*}
D_{\nu_{\bm r}^n}^-(\sqrt{f})
 &= \frac{2}{n} \int \sum_{k=1}^{2} \left[  \alpha_k \eta_0^1 \bigl(\sqrt{f}(\eta)-\sqrt{f}( \eta - \delta_0^1 +\delta_k^1)\bigr) \right. \\ 
 & \left. \quad + \alpha_0 \eta_k^1 \bigl(\sqrt{f}(\eta)-\sqrt{f}(\eta-\delta_k^1 + \delta_0^1 )\bigr) \right]  \sqrt{f}(\eta)\, \mathrm d\nu_{\bm r}^n
= -2 \langle L_n^- \sqrt{f}, \sqrt{f} \rangle_{ \nu_{\bm r}^n}.	 
\end{align*}
The proof for \(L_n^+\) is the same, using the relations at the right boundary.

 We now consider the conservative part. For \(1\le x\le n-2\) and \(k,l\in\{0,1,2\}\), we write
 \[
 \eta^{x,x+1}_{k,l}
 =
 \eta-\delta_k^x+\delta_l^x+\delta_k^{x+1}-\delta_l^{x+1},
 \]
 and 
 \[
 R_{k,l}^{x,x+1}
 =
 \frac{
 	r_k(x/n)r_l((x+1)/n)
 }{
 	r_l(x/n)r_k((x+1)/n)
 }.
 \]
 Using the identity
 \[
 \sqrt a(\sqrt b-\sqrt a)
 =
 -\frac14(\sqrt b-\sqrt a)^2
 +\frac14 b-\frac14 a
 +\frac12\sqrt a(\sqrt b-\sqrt a),
 \]
 with \(a=f(\eta)\) and \(b=f(\eta_{k,l}^{x,x+1})\), we obtain
 \begin{align*}
 	\left\langle L_n^s \sqrt f,\sqrt f\right\rangle_{\nu_{\bm r}^n}
 	&=  \int   \sum_{x=1}^{n-2} \sum_{k,l=0}^{2} \eta_k^x \eta_l^{x+1} 
 	\bigl[ \sqrt{f}(\eta_{k,l}^{x,x+1}) - \sqrt{f}(\eta)  \bigr]\sqrt{f}(\eta) \,  \mathrm d\nu_{\bm r}^n \\
 	&=
 	-\frac14D_{\nu_{\bm r}^n}^s(\sqrt{f})    
 	+\frac14
 	\int
 	\sum_{x=1}^{n-2}
 	\sum_{k,l=0}^2
 	\eta_k^x\eta_l^{x+1}
 	\left[
 	f(\eta_{k,l}^{x,x+1})-f(\eta)
 	\right]
 	\,\mathrm d\nu_{\bm r}^n     \\
 	&\quad
 	+\frac12
 	\int
 	\sum_{x=1}^{n-2}
 	\sum_{k,l=0}^2
 	\eta_k^x\eta_l^{x+1}
 	\sqrt{f}(\eta)
 	\left[
 	\sqrt{f}(\eta_{k,l}^{x,x+1})-\sqrt{f}(\eta)
 	\right]
 	\,\mathrm d\nu_{\bm r}^n .
 \end{align*}
 For the last term, we perform the change of variables $ \eta^{x,x+1}_{k,l} = \xi $. Then
 \[
 \begin{aligned}
 	&\frac12
 	\int
 	\sum_{x=1}^{n-2}
 	\sum_{k,l=0}^2
 	\eta_k^x\eta_l^{x+1}
 	\sqrt{f}(\eta)
 	\left[
 	\sqrt{f}(\eta_{k,l}^{x,x+1})-\sqrt{f}(\eta)
 	\right]
 	\,\mathrm d\nu_{\bm r}^n       \\
 	&= \frac12
 	\int
 	\sum_{x=1}^{n-2}
 	\sum_{k,l=0}^2
 	\eta_k^x\eta_l^{x+1} 
 	\left[ \sqrt{f}(\eta)
 	\sqrt{f}\left(
 	\eta_{k,l}^{x,x+1}
 	\right)
 	-
 	f(\eta_{k,l}^{x,x+1})
 	\right]R_{k,l}^{x,x+1}
 	\,\mathrm d\nu_{\bm r}^n \\
 	&=-\frac14
 	\int
 	\sum_{x=1}^{n-2}
 	\sum_{k,l=0}^2
 	\eta_k^x\eta_l^{x+1}
 	\left[
 	\sqrt{f}(\eta_{k,l}^{x,x+1})-\sqrt{f}(\eta)
 	\right]^2
 	R_{k,l}^{x,x+1}
 	\,\mathrm d\nu_{\bm r}^n      \\
 	&\quad
 	+\frac14
 	\int
 	\sum_{x=1}^{n-2}
 	\sum_{k,l=0}^2
 	\eta_k^x\eta_l^{x+1}
 	\left[
 	f(\eta) - f(\eta_{k,l}^{x,x+1}) 
 	\right]
 	R_{k,l}^{x,x+1}
 	\,\mathrm d\nu_{\bm r}^n .
 \end{aligned}
 \]
 Consequently,
 \[
 \begin{aligned}
 	 \left\langle L_n^s \sqrt{f},\sqrt{f}\right\rangle_{\nu_{\bm r}^n}  
 	&\le
 	- \frac{1}{4} \int
 	\sum_{x=1}^{n-2}
 	\sum_{k,l=0}^2
 	\eta_k^x\eta_l^{x+1}
 	\left[
 	\sqrt{f}(\eta_{k,l}^{x,x+1})-\sqrt{f}(\eta)
 	\right]^2
 	\left[ 1+R_{k,l}^{x,x+1} \right]  
 	\,\mathrm d\nu_{\bm r}^n  \\
 	&\quad 
 	+\frac14
 	\int
 	\sum_{x=1}^{n-2}
 	\sum_{k,l=0}^2
 	\eta_k^x\eta_l^{x+1}
 	\left[
 	f(\eta_{k,l}^{x,x+1})-f(\eta)
 	\right]
 	\left[
 	1-R_{k,l}^{x,x+1}
 	\right]
 	\,\mathrm d\nu_{\bm r}^n .
 \end{aligned}
 \]
 Using the algebraic identity $b-a=(\sqrt{b}-\sqrt{a})(\sqrt{b}+\sqrt{a})$ and Young's inequality, we obtain
 \begin{align*}
 	&\frac14
 	\int
 	\sum_{x=1}^{n-2}
 	\sum_{k,l=0}^{2}
 	\eta_{k}^{x}\eta_{l}^{x+1}
 	\Bigl[
 	f\bigl(\eta_{k,l}^{x,x+1}\bigr)-f(\eta)
 	\Bigr]
 	\Bigl[
 	1-R_{k,l}^{x,x+1}
 	\Bigr]
 	\,\mathrm d\nu_{\bm r}^{n}
 	\\
 	&\le
 	\frac18
 	\int
 	\sum_{x=1}^{n-2}
 	\sum_{k,l=0}^{2}
 	\eta_{k}^{x}\eta_{l}^{x+1}
 	\Bigl[
 	\sqrt{f}\bigl(\eta_{k,l}^{x,x+1}\bigr)-\sqrt{f}(\eta)
 	\Bigr]^{2}
 	\,\mathrm d\nu_{\bm r}^{n}
 	\\
 	&\quad
 	+
 	\frac14
 	\int
 	\sum_{x=1}^{n-2}
 	\sum_{k,l=0}^{2}
 	\eta_{k}^{x}\eta_{l}^{x+1}
 	\Bigl[
 	f\bigl(\eta_{k,l}^{x,x+1}\bigr)+f(\eta)
 	\Bigr]
 	\Bigl[
 	1-R_{k,l}^{x,x+1}
 	\Bigr]^{2}
 	\,\mathrm d\nu_{\bm r}^{n}.
 \end{align*}
  By the Lipschitz continuity of \(\bm r\) and the uniform lower bound \(r_i \ge a\),
 \[
 \left|1-R_{k,l}^{x,x+1}\right|
 \le
 \frac{C}{n}.
 \]
 Furthermore, since \(f\) is a density with respect to \(\nu_{\bm r}^n\) and $ 0 \le \eta_k^x \eta_l^{x+1} \le 1 $, the last integral is bounded by \(C/n\). Hence
 \[
 \left\langle L_n^s\sqrt f,\sqrt f\right\rangle_{\nu_{\bm r}^n}
 \le
 -MD_{\nu_{\bm r}^n}^s(\sqrt f)+\frac{C}{n}
 \]
 for some \(M>0\).
 
 Finally, for the conversion part, observe that only \(O(n)\) sites are involved and
 all conversion rates are uniformly bounded. Since the one-site Radon-Nikodym
 ratios are uniformly bounded by the assumptions on \(\bm r\), we have
 \[
 \left\langle L_n^c\sqrt f,\sqrt f\right\rangle_{\nu_{\bm r}^n}
 \le
 Cn.
 \]
 The proof is complete.

\end{proof}

We now introduce some notation for the local averages. For \(0<\varepsilon<1/2\), define
\[
\Sigma_n^{\varepsilon,L}
:=
\{1,\cdots,\lfloor\varepsilon(n-1)\rfloor\},
\qquad
\Sigma_n^{\varepsilon,R}
:=
\{n-1-\lfloor\varepsilon(n-1)\rfloor,\cdots,n-1\}.
\]
For \(\ell\in\mathbb N\), define the right and left local averages by
\[
\overrightarrow\eta_k^{\,\ell}(x)
=
\frac1\ell
\sum_{y=x+1}^{x+\ell}\eta_k^y,
\qquad
\overleftarrow\eta_k^{\,\ell}(x)
=
\frac1\ell
\sum_{y=x-\ell}^{x-1}\eta_k^y.
\]
These quantities are used only when the corresponding averaging boxes are contained
in \(\Sigma_n\). For a configuration \(\eta\in\Omega_n\) and $ x \in \Sigma_{n} $, we define the translation by $x$ of $ \eta $ as $ (\tau_x \eta)(y)=\eta(x+y)$.

\begin{lemma}[Replacement lemma]\label{l4.2}
	Let \(k\in\{1,2\}\). Fix \(x\notin\Sigma_n^{\varepsilon,R}\), and let
	\(\psi:\Omega_n\to\mathbb R\) be a uniformly bounded function whose support does not
	intersect the averaging box $ \{x+1,\ldots,x+\lfloor\varepsilon n\rfloor\} $.
	Then, for every \(t\in[0,T]\),
	\[
	\lim_{\varepsilon\downarrow0}
	\limsup_{n\to\infty}
	\mathbb E_{\mu^n}^n
	\left[
	\left|
	\int_0^t
	\psi(\tau_x\eta_s)
	\left(
	\eta_k^x(s)
	-
	\overrightarrow\eta_k^{\,\lfloor\varepsilon n\rfloor}(x,s)
	\right)
	\,\mathrm ds
	\right|
	\right]
	=0.
	\]
	The analogous estimate holds for the left average
	\(\overleftarrow\eta_k^{\,\ell}(x)\).
\end{lemma}

\begin{proof}
	The proof follows the strategy of Lemma~E.1 in \cite{ref6}. We prove the assertion for the right average; the proof for the left average is identical.
	Let $ \ell=\lfloor\varepsilon n\rfloor $ and  
	\[
	W(\eta)
	=
	\psi(\tau_x\eta)
	\left(
	\eta_k^x-\overrightarrow\eta_k^{\,\ell}(x)
	\right).
	\]
	
     By the entropy inequality, for every \(\theta>0\),
     \[
     \begin{aligned}
     	\mathbb E_{\mu^n}^n
     	\left[
     	\left|
     	\int_0^t W(\eta_s)\,\mathrm ds
     	\right|
     	\right]
     	&\le
     	\frac{H(\mu^n\,|\,\nu_{\bm r}^n)}{\theta n}
     	+
     	\frac1{\theta n}
     	\log
     	\mathbb E_{\nu_{\bm r}^n}
     	\left[
     	\exp\left\{
     	\theta n
     	\left|
     	\int_0^t W(\eta_s)\,\mathrm ds
     	\right|
     	\right\}
     	\right]      \\
     	&\le
     	\frac{H(\mu^n\,|\,\nu_{\bm r}^n)}{\theta n}
     	+
     	\frac1{\theta n}
     	\left(
     	\log2
     	+
     	\max_{\pm}
     	\log
     	\mathbb E_{\nu_{\bm r}^n}
     	\left[
     	\exp\left\{
     	\pm\theta n
     	\int_0^t W(\eta_s)\,\mathrm ds
     	\right\}
     	\right]
     	\right),
     \end{aligned}
     \]
     where we used the elementary bounds \(e^{|a|}\le e^a+e^{-a}\) and $ \log (a+b) \le \log 2 + \max \{ \log a, \log b \} $.
     By the Feynman-Kac formula,
     \[
     \begin{aligned}
     	&\mathbb E_{\mu^n}^n
     	\left[
     	\left|
     	\int_0^t W(\eta_s)\,\mathrm ds
     	\right|
     	\right] \\
     	&\le
     	\frac{H(\mu^n\,|\,\nu_{\bm r}^n)}{\theta n}
     	+
     	\frac1{\theta n}
     	\left(
     	\log2
     	+
     	\max_{\pm}
     	\int_0^t
     	\sup_f
     	\left\{
     	\pm\theta n
     	\int W(\eta)f(\eta)\,\mathrm d\nu_{\bm r}^n
     	+
     	\left\langle L_n\sqrt f,\sqrt f\right\rangle_{\nu_{\bm r}^n}
     	\right\}
     	\mathrm ds
     	\right),
     \end{aligned}
     \]
     where the supremum is taken over all densities \(f\) with respect to
     \(\nu_{\bm r}^n\).	
     
	Since \(\mu^n\) is a probability measure, we have
	\[
	\begin{aligned}
		H(\mu^n\,|\,\nu_{\bm r}^n)
		 =
		\sum_{\eta\in\Omega_n}
		\mu^n(\eta)
		\log\left(
		\frac{\mu^n(\eta)}{\nu_{\bm r}^n(\eta)}
		\right)      
		 \le
		\max_{\eta\in\Omega_n}
		\left\{
		\log\left[
		\nu_{\bm r}^n(\eta)
		\right]^{-1}
		\right\}.
	\end{aligned}
	\]
	By $ 0 < a\le r_i(u)\le b < 1 $, we have
	\[
	\begin{aligned}
		\log\left[
		\nu_{\bm r}^n(\eta)
		\right]^{-1}
		&=
		\sum_{x=1}^{n-1}
		\sum_{i=0}^2
		\mathbf 1_{\{\eta_i^x=1\}}
		\log
		r_i\left(\frac{x}{n}\right)^{-1}
		\le
		Cn.
	\end{aligned}
	\]
	Thus
	\[
	H(\mu^n\,|\,\nu_{\bm r}^n)\le Cn.
	\]
	Using Lemma~\ref{l4.1}, we obtain
	\[
	\begin{aligned}
		\mathbb E_{\mu^n}^n
		\left[
		\left|
		\int_0^t W(\eta_s)\,\mathrm ds
		\right|
		\right]
		&\le
		\frac{C}{\theta}
		+
		\max_{\pm}
		\int_0^t
		\sup_f
		\Bigg\{
		\pm
		\int W(\eta)f(\eta)\,\mathrm d\nu_{\bm r}^n       \\
		&\quad
		-
		\frac{nM}{\theta}
		D_{\nu_{\bm r}^n}^s(\sqrt f)
		-
		\frac{n}{2\theta}
		D_{\nu_{\bm r}^n}^+(\sqrt f)
		-
		\frac{n}{2\theta}
		D_{\nu_{\bm r}^n}^-(\sqrt f)
		\Bigg\}
		\,\mathrm ds.
	\end{aligned}
	\]

	It remains to estimate $ \int W(\eta)f(\eta)\,\mathrm d\nu_{\bm r}^n $. Note that
	\[
	\begin{aligned}
		\int W(\eta)f(\eta)\,d\nu_{\bm r}^n
		&=
		\frac1\ell
		\sum_{y=x+1}^{x+\ell}
		\sum_{w=x}^{y-1}
		\int
		\left(
		\eta_k^w-\eta_k^{w+1}
		\right)
		\psi(\tau_x\eta)f(\eta)
		\,\mathrm d\nu_{\bm r}^n.
	\end{aligned}
	\]
	Since $ \eta_0^x+\eta_1^x+\eta_2^x=1  $ for every $ x \in\Sigma_n $, we have
	\[
	\eta_k^w-\eta_k^{w+1}
	=
	\eta_k^w\sum_{\substack{l=0\\l\ne k}}^2\eta_l^{w+1}
	-
	\eta_k^{w+1}\sum_{\substack{l=0\\l\ne k}}^2\eta_l^w.
	\]
	For \(l\ne k\), define $ \eta_{k,l}^{w,w+1}
	=
	\eta-\delta_k^w+\delta_l^w+\delta_k^{w+1}-\delta_l^{w+1} $.
	Then
	\[
	\begin{aligned}
		 \int
		\left(
		\eta_k^w-\eta_k^{w+1}
		\right)
		\psi(\tau_x\eta)f(\eta)
		\,\mathrm d\nu_{\bm r}^n        
		&=
		\frac12
		\int
		\sum_{\substack{l=0\\l\ne k}}^2
		\psi(\tau_x\eta)
		\eta_k^w\eta_l^{w+1}
		\left[
		f(\eta)-f(\eta_{k,l}^{w,w+1})
		\right]
		\,\mathrm d\nu_{\bm r}^n       \\
		&\quad
		+
		\frac12
		\int
		\sum_{\substack{l=0\\l\ne k}}^2
		\psi(\tau_x\eta)
		\eta_k^w\eta_l^{w+1}
		\left[
		f(\eta)+f(\eta_{k,l}^{w,w+1})
		\right]
		\,\mathrm d\nu_{\bm r}^n       \\
		&\quad
		-
		\frac12
		\int
		\sum_{\substack{l=0\\l\ne k}}^2
		\psi(\tau_x\eta)
		\eta_k^{w+1}\eta_l^w
		\left[
		f(\eta)-f(\eta_{l,k}^{w,w+1})
		\right]
		\,\mathrm d\nu_{\bm r}^n       \\
		&\quad
		-
		\frac12
		\int
		\sum_{\substack{l=0\\l\ne k}}^2
		\psi(\tau_x\eta)
		\eta_k^{w+1}\eta_l^w
		\left[
		f(\eta)+f(\eta_{l,k}^{w,w+1})
		\right]
		\,\mathrm d\nu_{\bm r}^n       \\
		&=:A_1+A_2+A_3+A_4.
	\end{aligned}
	\]
	We first estimate \(A_2+A_4\). We perform respectively the changes of variables $ \eta_{k,l}^{w,w+1} = \xi $ and $ \eta_{l,k}^{w,w+1} = \zeta $.
	Since the support of \(\psi\) does not intersect the averaging box
	\(\{x+1,\ldots,x+\ell\}\), we have
	\[
	\psi(\tau_x\eta)=\psi(\tau_x\xi),
	\qquad
	\psi(\tau_x\eta)=\psi(\tau_x\zeta).
	\]
	Hence
	\begin{align*}
		A_2+A_4
		&=
		\frac12
		\int
		\sum_{\substack{l=0\\l\ne k}}^2
		\psi(\tau_x\eta)
		\eta_k^w\eta_l^{w+1}
		f(\eta)
		\,\mathrm d\nu_{\bm r}^n       
		-
		\frac12
		\int
		\sum_{\substack{l=0\\l\ne k}}^2
		\psi(\tau_x\eta)
		\eta_k^{w+1}\eta_l^w
		f(\eta)
		\,\mathrm d\nu_{\bm r}^n       \\
		&\quad
		+
		\frac12
		\sum_{\substack{l=0\\l\ne k}}^2
		\int
		\psi(\tau_x\xi)
		\xi_l^w\xi_k^{w+1}
		f(\xi)
		\frac{
			r_k(w/n)r_l((w+1)/n)
		}{
			r_l(w/n)r_k((w+1)/n)
		}
		\,\mathrm d\nu_{\bm r}^n       \\
		&\quad
		-
		\frac12
		\sum_{\substack{l=0\\l\ne k}}^2
		\int
		\psi(\tau_x\zeta)
		\zeta_l^{w+1}\zeta_k^w
		f(\zeta)
		\frac{
			r_l(w/n)r_k((w+1)/n)
		}{
			r_k(w/n)r_l((w+1)/n)
		}
		\,\mathrm d\nu_{\bm r}^n       \\
		&=
		\frac12
		\sum_{\substack{l=0\\l\ne k}}^2
		\int
		\psi(\tau_x\eta)
		\eta_k^w\eta_l^{w+1}
		f(\eta)
		\left[
		1-
		\frac{
			r_l(w/n)r_k((w+1)/n)
		}{
			r_k(w/n)r_l((w+1)/n)
		}
		\right]
		\,\mathrm d\nu_{\bm r}^n       \\
		&\quad
		+
		\frac12
		\sum_{\substack{l=0\\l\ne k}}^2
		\int
		\psi(\tau_x\eta)
		\eta_k^{w+1}\eta_l^w
		f(\eta)
		\left[
		\frac{
			r_k(w/n)r_l((w+1)/n)
		}{
			r_l(w/n)r_k((w+1)/n)
		}
		-1
		\right]
		\,\mathrm d\nu_{\bm r}^n .
	\end{align*}
    Since \(\bm r\) is Lipschitz continuous and uniformly bounded away from zero,
    \[
    \left|
    1-
    \frac{
    	r_l(w/n)r_k((w+1)/n)
    }{
    	r_k(w/n)r_l((w+1)/n)
    }
    \right|
    \le
    \frac{C}{n}.
    \]
    Therefore,
    \[
    |A_2+A_4|
    \le
    \frac{C}{n}\|\psi\|_\infty.
    \]
	
	 We now estimate \(A_1+A_3\). Using $ a-b=(\sqrt a-\sqrt b)(\sqrt a+\sqrt b) $, we have
	 \[
	 \begin{aligned}
	 	A_1+A_3
	 	&=
	 	\frac12
	 	\int
	 	\sum_{\substack{l=0\\l\ne k}}^2
	 	\psi(\tau_x\eta)
	 	\eta_k^w\eta_l^{w+1}
	 	\left[
	 	\sqrt f(\eta)-\sqrt f(\eta_{k,l}^{w,w+1})
	 	\right]   
	 	\left[
	 	\sqrt f(\eta)+\sqrt f(\eta_{k,l}^{w,w+1})
	 	\right]
	 	\,\mathrm d\nu_{\bm r}^n       \\
	 	&\quad
	 	-
	 	\frac12
	 	\int
	 	\sum_{\substack{l=0\\l\ne k}}^2
	 	\psi(\tau_x\eta)
	 	\eta_k^{w+1}\eta_l^w
	 	\left[
	 	\sqrt f(\eta)-\sqrt f(\eta_{l,k}^{w,w+1})
	 	\right]        
	 	\left[
	 	\sqrt f(\eta)+\sqrt f(\eta_{l,k}^{w,w+1})
	 	\right]
	 	\,\mathrm d\nu_{\bm r}^n .
	 \end{aligned}
	 \]
	 By Young's inequality, for every \(\sigma>0\), 
	 we obtain
	 \[
	 \begin{aligned}
	 	A_1+A_3
	 	&\le
	 	\frac1{4\sigma}
	 	\sum_{\substack{l=0\\l\ne k}}^2
	 	\int
	 	\eta_k^w\eta_l^{w+1}
	 	\left[
	 	\sqrt f(\eta)-\sqrt f(\eta_{k,l}^{w,w+1})
	 	\right]^2
	 	\,\mathrm d\nu_{\bm r}^n       \\
	 	&\quad
	 	+
	 	\frac{\sigma}{2}
	 	\sum_{\substack{l=0\\l\ne k}}^2
	 	\int
	 	\psi(\tau_x\eta)^2
	 	\eta_k^w\eta_l^{w+1}
	 	\left[
	 	f(\eta)+f(\eta_{k,l}^{w,w+1})
	 	\right]
	 	\,\mathrm d\nu_{\bm r}^n       \\
	 	&\quad
	 	+
	 	\frac1{4\sigma}
	 	\sum_{\substack{l=0\\l\ne k}}^2
	 	\int
	 	\eta_l^w\eta_k^{w+1}
	 	\left[
	 	\sqrt f(\eta)-\sqrt f(\eta_{l,k}^{w,w+1})
	 	\right]^2
	 	\,\mathrm d\nu_{\bm r}^n       \\
	 	&\quad
	 	+
	 	\frac{\sigma}{2}
	 	\sum_{\substack{l=0\\l\ne k}}^2
	 	\int
	 	\psi(\tau_x\eta)^2
	 	\eta_l^w\eta_k^{w+1}
	 	\left[
	 	f(\eta)+f(\eta_{l,k}^{w,w+1})
	 	\right]
	 	\,\mathrm d\nu_{\bm r}^n .
	 \end{aligned}
	 \]
	 Consequently,
	 \[
	 \begin{aligned}
	 	&\int
	 	\left(
	 	\eta_k^w-\eta_k^{w+1}
	 	\right)
	 	\psi(\tau_x\eta)f(\eta)
	 	\,\mathrm d\nu_{\bm r}^n      \\
	 	&\le
	 	\frac1{4\sigma}
	 	\sum_{\substack{l=0\\l\ne k}}^2
	 	\left(
	 	D_{k,l}^{w,w+1}(\sqrt f)
	 	+
	 	D_{l,k}^{w,w+1}(\sqrt f)
	 	\right)
	 	+
	 	C\sigma\|\psi\|_\infty^2
	 	+
	 	\frac{C}{n}\|\psi\|_\infty,
	 \end{aligned}
	 \]
	 where
	 \[
	 D_{k,l}^{w,w+1}(\sqrt f)
	 =
	 \int
	 \eta_k^w\eta_l^{w+1}
	 \left[
	 \sqrt f(\eta)-\sqrt f(\eta_{k,l}^{w,w+1})
	 \right]^2
	 \,\mathrm d\nu_{\bm r}^n .
	 \]
	 Combining the above estimate, we obtain
	 \begin{align*}
	 	&\pm
	 	\int W(\eta)f(\eta)\,\mathrm d\nu_{\bm r}^n
	 	-
	 	\frac{nM}{\theta}
	 	D_{\nu_{\bm r}^n}^s(\sqrt f)
	 	-
	 	\frac{n}{2\theta}
	 	D_{\nu_{\bm r}^n}^+(\sqrt f)
	 	-
	 	\frac{n}{2\theta}
	 	D_{\nu_{\bm r}^n}^-(\sqrt f)      \\
	 	&\le
	 	\pm
	 	\frac1\ell
	 	\sum_{y=x+1}^{x+\ell}
	 	\sum_{w=x}^{y-1}
	 	\left[
	 	\frac1{4\sigma}
	 	\sum_{\substack{l=0\\l\ne k}}^2
	 	\left(
	 	D_{k,l}^{w,w+1}(\sqrt f)
	 	+
	 	D_{l,k}^{w,w+1}(\sqrt f)
	 	\right)
	 	\right]        
	 	-
	 	\frac{nM}{\theta}
	 	D_{\nu_{\bm r}^n}^s(\sqrt f)
	 	+
	 	C\sigma\ell
	 	+
	 	\frac{C\ell}{n}        \\
	 	&\le
	 	\left(
	 	\frac1{4\sigma}
	 	-
	 	\frac{nM}{\theta}
	 	\right)
	 	D_{\nu_{\bm r}^n}^s(\sqrt f)
	 	+
	 	C\sigma\ell
	 	+
	 	\frac{C\ell}{n}.
	 \end{align*}
	 Choosing $ \sigma=\frac{\theta}{4nM} $, 
	 we obtain
	 \[
	 \limsup_{n\to\infty}
	 \mathbb E_{\mu^n}^n
	 \left[
	 \left|
	 \int_0^t W(\eta_s)\,\mathrm ds
	 \right|
	 \right]
	 \le
	 \frac{C}{\theta}
	 +
	 C\theta\varepsilon
	 +
	 C\varepsilon.
	 \]
	 Letting first \(\varepsilon\downarrow0\) and then \(\theta\to\infty\), we conclude that
	 \[
	 \lim_{\varepsilon\downarrow0}
	 \limsup_{n\to\infty}
	 \mathbb E_{\mu^n}^n
	 \left[
	 \left|
	 \int_0^t W(\eta_s)\,\mathrm ds
	 \right|
	 \right]
	 =
	 0.
	 \]
	 This proves the replacement lemma for the right average. The proof for the left
	 average is identical.

\end{proof}

\section{Proof of Theorem~\ref{th2.2}}

Let \(k \in \{1,2 \} \), and let $ f_k :[0,T]\times[0,1]\to\mathbb R $
be a smooth time-dependent test function. We first write the Dynkin martingale
associated with the \(k\)-th fluctuation field:
\[
M_k^{n,t}(f_k^\cdot)
=
Y_k^{n,t}(f_k^t)
-
Y_k^{n,0}(f_k^0)
-
\int_0^t
(\partial_s+L_n)Y_k^{n,s}(f_k^s)\,\mathrm ds .
\]
For \(k,l\in\{1,2\}\), the corresponding Doob martingale is
\[
N_{k,l}^{n,t}(f_k,f_l)
=
M_k^{n,t}(f_k)M_l^{n,t}(f_l)
-
\int_0^t
\Gamma_n
\left(
Y_k^{n,s}(f_k^s),
Y_l^{n,s}(f_l^s)
\right)
\,\mathrm ds,
\]
where the carr\'e du champ operator is defined by
\[
\Gamma_n(F,G)
=
L_n(FG)-F L_nG-G L_nF .
\]
Consequently,
\[
\left\langle
M_k^{n}(f_k),M_l^{n}(f_l)
\right\rangle_t
=
\int_0^t
\Gamma_n(f_k,f_l)\,\mathrm ds .
\]
For a vector-valued test function $ \bm f=(f_1,f_2)^\top $, we define the coupled martingale by
\[
\bm M_t^n(\bm f)
=
M_1^{n,t}(f_1^\cdot)
+
M_2^{n,t}(f_2^\cdot).
\]

For \(k \in \{1,2\}\), we write \(3-k\) for the other species. A direct computation gives
\[
\begin{aligned}
	(\partial_s+L_n)Y_k^{n,s}(f_k^s)
	&=
	\frac1{\sqrt n}
	\sum_{x=1}^{n-1}
	\partial_s f_k^s\left(\frac{x}{n}\right)
	\bar\eta_k^s(x)      
	+
	\frac1{\sqrt n}
	\sum_{x=1}^{n-1}
	\Delta_n f_k^s\left(\frac{x}{n}\right)
	\bar\eta_k^s(x)       \\
	&\quad
	+
	\frac1{\sqrt n}
	\sum_{x=2}^{n-2}
	\left(
	\gamma_{3-k}\bar\eta_{3-k}^s(x)
	-
	\gamma_k\bar\eta_k^s(x)
	\right)
	f_k^s\left(\frac{x}{n}\right)
	+
	R_k^n(f_k^s),
\end{aligned}
\]
where $ \bar\eta_k^s(x)=\eta_k^s(x)-\rho_k^{n,s}(x) $,
and
\begin{align}\label{eq5.1}
	R_k^n(f_k^s)
	&=
	\frac{n^2}{\sqrt n}
	\left(
	f_k^s\left(\frac1n\right)-f_k^s(0)
	\right)
	\bar\eta_k^s(1)        
	+
	\frac{n^2}{\sqrt n}
	\left(
	f_k^s\left(\frac{n-1}{n}\right)-f_k^s(1)
	\right)
	\bar\eta_k^s(n-1)   \notag   \\
	&\quad
	+
	\sqrt n
	\left(
	\alpha_k\bar\eta_0^s(1)
	-
	\alpha_0\bar\eta_k^s(1)
	\right)
	f_k^s\left(\frac1n\right)      
	+
	\sqrt n
	\left(
	\beta_k\bar\eta_0^s(n-1)
	-
	\beta_0\bar\eta_k^s(n-1)
	\right)
	f_k^s\left(\frac{n-1}{n}\right).
\end{align}
Here
\[
\Delta_n f\left(\frac{x}{n}\right)
=
n^2
\left[
f\left(\frac{x+1}{n}\right)
+
f\left(\frac{x-1}{n}\right)
-
2f\left(\frac{x}{n}\right)
\right].
\]

The carr\'e du champ can also be computed explicitly. For \(k=1,2\),
\begin{align}\label{eq5.2}
	\Gamma_n\left( Y_k^{n,s}(f_k^s),\, Y_k^{n,s}(f_k^s)\right) 
	&= 
	n\sum_{x=1}^{n-2}
	\left(\eta_k^{s }(x)-\eta_k^{s }(x+1)\right)^2
	\left( 
	f^s_k\left(\frac{x+1}{n}\right)
	-
	f^s_k\left(\frac{x}{n}\right)
	\right) ^2  \notag \\
	&+
	\frac1n\sum_{x=2}^{n-2}
	\left(
	\gamma_k\eta_k^{s }(x)
	+
	\gamma_{3-k}\eta_{3-k}^{s }(x)
	\right)
	f^s_k\left(\frac{x}{n}\right)^2
	\nonumber\\
	&+
	\left[
	\alpha_k\eta_0^{s }(1)
	+
	\alpha_0 \eta_k^{s }(1)
	\right]
	f^s_k\left(\frac1n\right)^2
	\nonumber\\
	&+
	\left[
	\beta_k\eta_0^{s }(n-1)
	+
	\beta_0\eta_k^{s }(n-1)
	\right]
	f^s_k\left(\frac{n-1}{n}\right)^2 .
\end{align}
For \(k\neq l\), we have
\begin{align}\label{eq5.3}
	\Gamma_n
	\left(
	Y_k^{n,s}(f_k^s),
	Y_l^{n,s}(f_l^s)
	\right)                                           
	&=
	-
	n
	\sum_{x=1}^{n-2}
	\left(
	\eta_k^s(x)\eta_l^{s}( x+1)
	+
	\eta_l^s(x)\eta_k^{s}(x+1)
	\right)                                    \notag      \\
	&\quad\times
	\left(
	f_k^s\left(\frac{x+1}{n}\right)
	-
	f_k^s\left(\frac{x}{n}\right)
	\right)
	\left(
	f_l^s\left(\frac{x+1}{n}\right)
	-
	f_l^s\left(\frac{x}{n}\right)
	\right)                                    \notag      \\
	&\quad
	-
	\frac1n
	\sum_{x=2}^{n-2}
	\left(
	\gamma_k\eta_k^s(x)
	+
	\gamma_l\eta_l^s(x)
	\right)
	f_k^s\left(\frac{x}{n}\right)
	f_l^s\left(\frac{x}{n}\right).
\end{align}

\begin{proposition}\label{p5.1}
	Let \(\bm f=(f_1,f_2)^\top\in\mathscr S^\dagger \). The sequence of martingales $ \left\{
	\bm M_t^n(\bm f):0\le t\le T \right\}_{n\in \mathbb{N} }$ converges in 
	the topology of  $ \mathcal D([0,T], \mathbb{R})$, as $ n  \to \infty $, towards a mean-zero Gaussian process \(W_t(\bm f)\) with quadratic variation given by
	\[
	\Phi_{11}+\Phi_{12}+\Phi_{21}+\Phi_{22},
	\]
	where, for \(k=1,2\),
	\[
	\begin{aligned}
		\Phi_{kk}
		&=
		\int_0^t
		\Bigg\{
		\int_0^1
		2\rho_k(s,u)
		\left(
		1-\rho_k(s,u)
		\right)
		\left(
		\partial_u f_k(u)
		\right)^2
		\,\mathrm du                                     
		+
		\int_0^1
		\left(
		\gamma_1\rho_1(s,u)
		+
		\gamma_2\rho_2(s,u)
		\right)
		f_k(u)^2
		\,\mathrm du                                      \\
		&\quad
		+
		\left(
		\alpha_k\rho_0(s,0)
		+
		\alpha_0\rho_k(s,0)
		\right)
		f_k(0)^2                                  
		+
		\left(
		\beta_k\rho_0(s,1)
		+
		\beta_0\rho_k(s,1)
		\right)
		f_k(1)^2
		\Bigg\}
		\,\mathrm ds,
	\end{aligned}
	\]
	and, for \(k\neq l\),
	\[
	\begin{aligned}
		\Phi_{kl}
		&=
		-
		\int_0^t
		\Bigg\{
		\int_0^1
		2\rho_1(s,u)\rho_2(s,u)
		\partial_u f_1(u)
		\partial_u f_2(u)
		\,\mathrm du                                      \\
		&\quad
		+
		\int_0^1
		\left(
		\gamma_1\rho_1(s,u)
		+
		\gamma_2\rho_2(s,u)
		\right)
		f_1(u)f_2(u)
		\,du
		\Bigg\}
		\,\mathrm ds .
	\end{aligned}
	\]
\end{proposition}

\begin{proof}
	We follow the proof of \cite[Lemma 3.2]{ref6} and apply \cite[Theorem VIII.3.12]{ref4}. The jumps of $ \left\{
	\bm M_t^n(\bm f):0\le t\le T
	\right\} $
	are uniformly bounded by $ \frac{C}{\sqrt n}\|\bm f\|_\infty $,
	and therefore vanish as \(n\to\infty\). It remains to identify the limiting quadratic
	variation.
	
	We first consider the conservative contribution in \eqref{eq5.2}. The first
	term on the right-hand side of \eqref{eq5.2} can be rewritten as
	\begin{align}\label{eq5.4}
		&
		\frac1n
		\sum_{x=1}^{n-2}
		\eta_k^s(x)
		\left(
		1-\eta_k^{s}(x+1 )
		\right)
		\left(
		\nabla_n^+ f_k\left(\frac{x}{n}\right)
		\right)^2                                      
		+
		\frac1n
		\sum_{x=1}^{n-2}
		\eta_k^{s}(x+1)
		\left(
		1-\eta_k^s(x)
		\right)
		\left(
		\nabla_n^+ f_k\left(\frac{x}{n}\right)
		\right)^2,
	\end{align}
	where
	\[
	\nabla_n^+ f_k\left(\frac{x}{n}\right)
	=
	n
	\left[
	f_k\left(\frac{x+1}{n}\right)
	-
	f_k\left(\frac{x}{n}\right)
	\right].
	\]
	We consider the first term $ \int_0^t
	\frac1n
	\sum_{x=1}^{n-2}
	\eta_k^s(x)
	\left(
	1-\eta_k^{s}(x+1)
	\right)
	\left(
	\nabla_n^+ f_k\left(\frac{x}{n}\right)
	\right)^2
	\,\mathrm ds $. 
	Using the notation introduced in the replacement lemma, we decompose the sum into the
	bulk region $	x\notin \Sigma_n^{\varepsilon,L}\cup\Sigma_n^{\varepsilon,R} $ and its complement. The contribution of the complement is uniformly bounded by \(C\varepsilon\).
    
    In the bulk region, applying the replacement lemma \ref{l4.2} twice, with suitable choices of
    the local function \(\psi\), gives
    \[
    \begin{aligned}
    	&
    	\int_0^t
    	\frac1n
    	\sum_{x\notin\Sigma_n^{\varepsilon,L}\cup\Sigma_n^{\varepsilon,R}}
    	\eta_k^s(x)
    	\left(
    	1-\eta_k^{s}(x+1)
    	\right)
    	\left(
    	\nabla_n^+ f_k\left(\frac{x}{n}\right)
    	\right)^2
    	\,\mathrm ds                                               \\
    	&=
    	\int_0^t
    	\frac1n
    	\sum_{x\notin\Sigma_n^{\varepsilon,L}\cup\Sigma_n^{\varepsilon,R}}
    	\overleftarrow\eta_k^{\,\lfloor\varepsilon n\rfloor}(x,s)
    	\left(
    	1-
    	\overrightarrow\eta_k^{\,\lfloor\varepsilon n\rfloor}(x+1,s)
    	\right)
    	\left(
    	\nabla_n^+ f_k\left(\frac{x}{n}\right)
    	\right)^2
    	\,\mathrm ds.
    \end{aligned}
    \]
      Define
      \[
      \iota_{\varepsilon }^{x/n}(u)
      =
      \frac{n}{\lfloor\varepsilon n\rfloor}
      \mathbf 1_{\left(x/n,(x + \lfloor\varepsilon n\rfloor)/n\right]}(u),
      \qquad
      \hbar_{\varepsilon }^{x/n}(u)
      =
      \frac{n}{ \lfloor\varepsilon n\rfloor }
      \mathbf 1_{\left[(x - \lfloor\varepsilon n\rfloor)/n,x/n\right)}(u).
      \]
      Then
      \[
      \overrightarrow\eta_k^{\,\lfloor\varepsilon n\rfloor}(x,s)
      =
      X_k^{n,s}\left(\iota_\varepsilon^{x/n}\right),
      \qquad
      \overleftarrow\eta_k^{\,\lfloor\varepsilon n\rfloor}(x,s)
      =
      X_k^{n,s}\left(\hbar_\varepsilon^{x/n}\right).
      \]
       By a standard smooth approximation argument, using the boundedness of the
       occupation variables, the hydrodynamic limit extends to these block test
       functions. Hence, as \(n\to\infty\),
      \[
      \begin{aligned}
      	&
      	\int_0^t
      	\frac1n
      	\sum_{x\notin\Sigma_n^{\varepsilon,L}\cup\Sigma_n^{\varepsilon,R}}
      	\overleftarrow\eta_k^{\,\lfloor\varepsilon n\rfloor}(x,s)
      	\left(
      	1-
      	\overrightarrow\eta_k^{\,\lfloor\varepsilon n\rfloor}(x+1,s)
      	\right)
      	\left(
      	\nabla_n^+ f_k\left(\frac{x}{n}\right)
      	\right)^2
      	\,\mathrm ds                                               \\
      	&\to
      	\int_0^t
      	\int_\varepsilon^{1-\varepsilon}
      	\left(f_k'(u)\right)^2
      	\left(
      	\frac1\varepsilon
      	\int_{u-\varepsilon}^{u}
      	\rho_k(s,v)\,\mathrm dv
      	\right)
      	\left(
      	1-
      	\frac1\varepsilon
      	\int_u^{u+\varepsilon}
      	\rho_k(s,v)\,\mathrm dv
      	\right)
      	\,\mathrm du\,\mathrm ds .
      \end{aligned}
      \]
      Letting \(\varepsilon  \to 0\), we obtain
      \[
      \begin{aligned}
      	&
      	\int_0^t
      	\frac1n
      	\sum_{x=1}^{n-2}
      	\eta_k^s(x)
      	\left(
      	1-\eta_k^{s}(x+1)
      	\right)
      	\left(
      	\nabla_n^+ f_k\left(\frac{x}{n}\right)
      	\right)^2
      	\,\mathrm ds                                           \\
      	&\to
      	\int_0^t
      	\int_0^1
      	\left(f_k'(u)\right)^2
      	\rho_k(s,u)
      	\left(
      	1-\rho_k(s,u)
      	\right)
      	\,\mathrm du\,\mathrm ds .
      \end{aligned}
      \]
      The second term in \eqref{eq5.4} gives the same limit.
      Therefore,
      \[
      \begin{aligned}
      	&\int_0^t
      	n
      	\sum_{x=1}^{n-2}
      	\left(
      	\eta_k^s(x)-\eta_k^{s}(x+1)
      	\right)^2
      	\left(
      	f_k\left(\frac{x+1}{n}\right)
      	-
      	f_k\left(\frac{x}{n}\right)
      	\right)^2
      	\,\mathrm ds                                             \\
      	&\to
      	\int_0^t
      	\int_0^1
      	2\rho_k(s,u)
      	\left(
      	1-\rho_k(s,u)
      	\right)
      	\left(
      	\partial_u f_k(u)
      	\right)^2
      	\,\mathrm du\,\mathrm ds .
      \end{aligned}
      \]
      The remaining items can also be obtained using a similar method.  Since the limit is deterministic, the convergence in probability also holds. This completes the proof.

\end{proof}

We now complete the proof of Theorem~\ref{th2.2}.
The proof of tightness is postponed to Section~7. 
Fix \(t\in[0,T]\), and restrict the
processes to the time interval \([0,t]\). We choose the time-dependent test function
\[
\bm g(s,u)
=
\begin{pmatrix}
	g_1(s,u)\\
	g_2(s,u)
\end{pmatrix}
:=
(S_{t-s}\bm f)(u),
\qquad 0\le s\le t,
\]
where \((S_r)_{r\ge0}\) is the semigroup generated by the adjoint operator $ \mathcal A^\dagger=\Delta \mathrm I_2+M^\top $ in Section~3. Therefore
\[
\partial_s\bm g(s,u)
=
-\mathcal A^\dagger \bm g(s,u)=
\begin{pmatrix}
	-\Delta g_1+\gamma_1g_1-\gamma_1g_2\\[0.1cm]
	-\Delta g_2-\gamma_2g_1+\gamma_2g_2
\end{pmatrix}.
\]

By the direct computation of the generator, we obtain
\begin{equation}\label{eq:backward-test-drift}
	\begin{aligned}
		&(\partial_s+L_n)
		\left(
		Y_1^{n,s}(g_1)+Y_2^{n,s}(g_2)
		\right)   
	    =
		\sum_{k=1}^2
		Y_k^{n,s}\left((\Delta_n-\Delta)g_k\right)
		+
		R_0^n(s)
		+
		R_1^n(s)
		+
		R_2^n(s).
	\end{aligned}
\end{equation}
Here
\[
R_0^n(s)
=
\frac1{\sqrt n}
\sum_{x\in\{1,n-1\}}
\left(
\gamma_1\bar\eta_1^s(x)-\gamma_2\bar\eta_2^s(x)
\right)
\left(
g_1\left(s,\frac{x}{n}\right)
-
g_2\left(s,\frac{x}{n}\right)
\right),
\]
and \(R_k^n(s)\) denotes the boundary error term associated with the \(k\)-th
component in \eqref{eq5.1}.  

We claim that $ (\partial_s+L_n)
\left(
Y_1^{n,s}(g_1)+Y_2^{n,s}(g_2)
\right)
\to 0 $
as \(n\to\infty\).
We first consider the bulk discretization error. Since \(\bm g(s,\cdot)\in\mathscr S^\dagger \)
is smooth uniformly for \(0\le s\le t\), Taylor's expansion gives
\[
\Delta_n g_k\left(s,\frac{x}{n}\right)
-
\Delta g_k\left(s,\frac{x}{n}\right)
=
O(n^{-2})
\]
uniformly in \(x\) and \(s\). Hence $ Y_k^{n,s}\left((\Delta_n-\Delta)g_k\right)
= O(n^{-3/2})$.
Next, the term \(R_0^n(s)\) only involves the two boundary sites \(x=1\) and
\(x=n-1\). Since all occupation variables are uniformly bounded, and since
\(\bm g\) is bounded uniformly on \([0,t]\times[0,1]\), we have $ R_0^n(s)=O(n^{-1/2}) $.
It remains to treat the boundary terms. We discuss the left boundary; the
right boundary is handled in exactly the same way. The left boundary contribution is
\[
\begin{aligned}
	B_L^n(s)
	&=
	\sum_{k=1}^2
	\frac{n^2}{\sqrt n}
	\left(
	g_k\left(s,\frac1n\right)-g_k(s,0)
	\right)
	\bar\eta_k^s(1)   
	+
	\sum_{k=1}^2
	\sqrt n
	\left(
	\alpha_k\bar\eta_0^s(1)
	-
	\alpha_0\bar\eta_k^s(1)
	\right)
	g_k\left(s,\frac1n\right).
\end{aligned}
\]
Since $ \eta_0^x+\eta_1^x+\eta_2^x=1 $ and $ \rho_0^{n,s}(x)+\rho_1^{n,s}(x)+\rho_2^{n,s}(x)=1 $, we have $ \bar\eta_0^s(x)
=
-\bar\eta_1^s(x)-\bar\eta_2^s(x) $ for every $ x \in \Sigma_{n} $.
Therefore \(B_L^n(s)\) can be rewritten as
\begin{align}
	B_L^n(s)
	&=
	\sqrt n
	\left[
	n\left(
	g_1\left(s,\frac1n\right)-g_1(s,0)
	\right)
	-
	(1-\alpha_2)g_1\left(s,\frac1n\right)
	-
	\alpha_2g_2\left(s,\frac1n\right)
	\right]
	\bar\eta_1^s(1)
	\nonumber\\
	&\quad
	+
	\sqrt n
	\left[
	n\left(
	g_2\left(s,\frac1n\right)-g_2(s,0)
	\right)
	-
	\alpha_1g_1\left(s,\frac1n\right)
	-
	(1-\alpha_1)g_2\left(s,\frac1n\right)
	\right]
	\bar\eta_2^s(1).
	\label{eq:left-boundary-error}
\end{align}
By Taylor's expansion,
\[
n\left(
g_1\left(s,\frac1n\right)-g_1(s,0)
\right)
=
\partial_u g_1(s,0)+O(n^{-1}),
\]
\[
n\left(
g_2\left(s,\frac1n\right)-g_2(s,0)
\right)
=
\partial_u g_2(s,0)+O(n^{-1}),
\]
and boundary condition at \(u=0\):
\[
\partial_u g_1(s,0)
=
(1-\alpha_2)g_1(s,0)+\alpha_2g_2(s,0),
\]
\[
\partial_u g_2(s,0)
=
\alpha_1g_1(s,0)+(1-\alpha_1)g_2(s,0).
\]
Substituting these identities into \eqref{eq:left-boundary-error}, and using again
the smoothness of \(\bm g\), we obtain $ B_L^n(s)=O(n^{-1/2}) $. The right boundary is treated similarly.
Hence $ (\partial_s+L_n)
\left(
Y_1^{n,s}(g_1)+Y_2^{n,s}(g_2)
\right)
\to 0 $ as \(n\to\infty\).

We now apply the martingale decomposition with the particular choice $ \bm g(s,u)=S_{t-s}\bm f (u)$.
Since \(\bm g(t,u)=\bm f(u) \) and \(\bm g(0,u)=S_t\bm f(u) \), we have
\[
\begin{aligned}
	\bm M_t^n(\bm g)
	&=
	\left\langle \bm Y_t^n,\bm f\right\rangle
	-
	\left\langle \bm Y_0^n,S_t\bm f\right\rangle        
	-
	\int_0^t
	(\partial_s+L_n)
	\left(
	Y_1^{n,s}(g_1)+Y_2^{n,s}(g_2)
	\right)
	\,\mathrm ds .
\end{aligned}
\]
The last integral converges to zero as \(n\to\infty\). By the martingale convergence
proved in Proposition~\ref{p5.1}, the martingale
\(\bm M_t^n(\bm g)\) converges to a centered Gaussian random variable, denoted by
\(W_t(\bm f)\), whose variance is the limiting quadratic variation evaluated along
the backward test function \(S_{t-s}\bm f\).
Together with the tightness of \(\{ \mathbb Q_n\}_{n\in \mathbb{N} }\) in
\(\mathcal D([0,T],(\mathcal S^\dagger)')\) proved in Section~7, this implies that along every convergent subsequence of the fluctuation fields, every
limit point satisfies
\[
\left\langle \bm Y_t,\bm f\right\rangle
=
\left\langle \bm Y_0,S_t\bm f\right\rangle
+
W_t(\bm f).
\]
Moreover, the noise term \(W_t(\bm f)\) is uncorrelated with the initial field.
Indeed, let \((\mathcal F_t)_{t\ge0}\) denote the limiting filtration. Since
\(W_t(\bm f)\) is obtained as the limit of the martingale term and starts from
zero, we have
\[
\mathbb E\!\left[W_t(\bm f)\mid \mathcal F_0\right]=0 .
\]
Therefore, for every \(\bm f,\bm g\in \mathcal S^\dagger\),
\[
\mathbb E\!\left[
W_t(\bm f) \bm Y_0 ( \bm g )
\right]
=
\mathbb E \left[
\bm Y_0 ( \bm g )
\mathbb E \left[W_t(\bm f)\mid \mathcal F_0\right]
\right]
=0.
\]
Thus \(W_t(\bm f)\) and \(\left\langle \bm Y_0,\bm g\right\rangle\) are uncorrelated.
This proves Theorem~\ref{th2.2}.

\section{Proof of Theorem~\ref{th2.3}}

\subsection{Uniqueness of the Ornstein-Uhlenbeck process}

We first prove the uniqueness of the limiting Ornstein-Uhlenbeck process. Recall that $ \mathcal A^\dagger=\Delta \mathrm I_2+M^\top $
denotes the adjoint operator acting on the test function space \(\mathscr S^\dagger \), and
\((S_t)_{t\ge0}\) denotes the semigroup generated by \(\mathcal A^\dagger\).

\begin{proposition}\label{p6.1}
	There exists a unique random element \(\bm Y\) taking values in
	\(\mathcal D([0,T], (\mathscr S^\dagger )')\) such that, for every test function \(\bm f\in \mathscr S^\dagger \),
	\[
	W_t(\bm f)
	=
	\left\langle \bm Y_t,\bm f\right\rangle
	-
	\left\langle \bm Y_0,\bm f\right\rangle
	-
	\int_0^t
	\left\langle \bm Y_s,\mathcal A^\dagger\bm f\right\rangle
	\,\mathrm d s
	\]
	  and
	\[
	N_t(\bm f)
	=
	W_t(\bm f)^2
	-
	\int_0^t
	\mathcal Q_s(\bm f,\bm f)
	\,\mathrm d s
	\]
	are martingales with respect to the natural filtration. Moreover, the initial field
	\(\bm Y_0\) is a centered Gaussian field with covariance
	\[
	\mathbb E\left[
	\bm Y_0(\bm f)\bm Y_0(\bm g)
	\right]
	=
	\sigma(\bm f,\bm g),
	\qquad
	\bm f,\bm g\in\mathscr S^\dagger,
	\]
	where \(\sigma\) is the initial covariance form defined in Theorem \ref{th2.3}.
\end{proposition}

Before proving the proposition, we need the following elementary semigroup property.
\begin{lemma}\label{l6.1}
	For every \(\bm f\in\mathscr S^\dagger\),
	\[
	S_{t+\varepsilon}\bm f-S_t\bm f
	=
	\varepsilon \mathcal A^\dagger S_t\bm f
	+
	o(\varepsilon,t),
	\]
	where \(o(\varepsilon,t)\in\mathscr S^\dagger\), and
	\[
	\lim_{\varepsilon\downarrow 0}
	\sup_{0\le t\le T}
	\frac{\|o(\varepsilon,t)\|_m}{\varepsilon}
	=
	0,
	\qquad
	\forall\,m\ge0.
	\]
\end{lemma}
\begin{proof}
	By the analyticity of \((S_t)_{t\ge0}\),
	\[
	\frac{\mathrm d}{\mathrm d t}S_t\bm f
	=
	\mathcal A^\dagger S_t\bm f
	=
	S_t\mathcal A^\dagger\bm f.
	\]
	Therefore,
	\[
	S_{t+\varepsilon}\bm f-S_t\bm f
	=
	\int_t^{t+\varepsilon}
	\frac{\mathrm d}{\mathrm d r}S_r\bm f
	\,\mathrm d r
	=
	\int_t^{t+\varepsilon}
	\mathcal A^\dagger S_r\bm f
	\,\mathrm d r.
	\]
	Setting \(r=t+s\), we have
	\[
	S_{t+\varepsilon}\bm f-S_t\bm f
	=
	\int_0^\varepsilon
	\mathcal A^\dagger S_{t+s}\bm f
	\,\mathrm d s
	=
	\varepsilon\mathcal A^\dagger S_t\bm f
	+
	o(\varepsilon,t),
	\]
	where
	\[
	o(\varepsilon,t)
	=
	\int_0^\varepsilon
	\left(
	\mathcal A^\dagger S_{t+s}\bm f
	-
	\mathcal A^\dagger S_t\bm f
	\right)
	\,\mathrm d s.
	\]
	Since the map $ r\mapsto \mathcal A^\dagger S_r\bm f $ is continuous in \(\mathscr S^\dagger \), and is uniformly continuous on compact time intervals,
	we obtain
	\[
	\frac{\|o(\varepsilon,t)\|_m}{\varepsilon}
	\le
	\sup_{0\le s\le\varepsilon}
	\left\|
	\mathcal A^\dagger S_{t+s}\bm f
	-
	\mathcal A^\dagger S_t\bm f
	\right\|_m
	\to 0
	\]
	uniformly for \(t\in[0,T]\). This proves the lemma.
\end{proof}

\begin{proof}[Proof of Proposition~\ref{p6.1}]
	The proof follows the standard argument for generalized Ornstein-Uhlenbeck
	martingale problems (\cite{ref3,ref5}). For every \(s\ge0\) and $ \bm f \in \mathscr{S}^\dagger $, by Ito's Formula (see [\cite{ref10}, Theorem~3.3 and Corollary~3.3]) and the discussion above, the process $ \{ X_t^s (\bm f) ; t \ge s \} $ defined by  
	\[
	X_t^s(\bm f)
	=
	\exp\left\{
	\frac12
	\int_s^t
	 Q_r(\bm f,\bm f)
	\,\mathrm d r
	+
	i
	\left[
	\bm Y_t(\bm f)
	-
	\bm Y_s(\bm f)
	-
	\int_s^t
	\bm Y_r(\mathcal A^\dagger\bm f)
	\,\mathrm d r
	\right]
	\right\},
	\qquad t\ge s
	\]
	is a complex-valued martingale.
	Fix \(\mathcal T>0\). We shall prove that
	\[
	Z_t
	=
	\exp\left\{
	\frac12
	\int_0^t
	  Q_r
	\left(
	S_{\mathcal T-r}\bm f,
	S_{\mathcal T-r}\bm f
	\right)
	\,\mathrm d r
	+
	i\bm Y_t(S_{\mathcal T-t}\bm f)
	\right\},
	\qquad 0\le t\le\mathcal T
	\]
	is also a complex-valued martingale.
	
	Let \(0\le t_1<t_2\le\mathcal T\). We divide the interval \([t_1,t_2]\) into \(N\)
	subintervals:
	\[
	t_1=s_0<s_1<\cdots<s_N=t_2,
	\qquad
	s_{j+1}-s_j=\frac{t_2-t_1}{N}.
	\]
	We have
	\begin{align*}
		\prod_{j=0}^{N-1}
		X_{s_{j+1}}^{s_j}(S_{\mathcal T-s_j}\bm f)
		&= \exp\Bigg\{
		\sum_{j=0}^{N-1}
		\frac12
		\int_{s_j}^{s_{j+1}}
		 Q_r
		\left(
		S_{\mathcal T-s_j}\bm f,
		S_{\mathcal T-s_j}\bm f
		\right)
		\,\mathrm d r                                      \\
		&\quad
		+
		i\sum_{j=0}^{N-1}
		\Bigg[
		\bm Y_{s_{j+1}}(S_{\mathcal T-s_j}\bm f)
		-
		\bm Y_{s_j}(S_{\mathcal T-s_j}\bm f)
		-
		\int_{s_j}^{s_{j+1}}
		\bm Y_r(\mathcal A^\dagger S_{\mathcal T-s_j}\bm f)
		\,\mathrm d r
		\Bigg]
		\Bigg\}.
	\end{align*}
     By the continuity of \(r\mapsto S_r\bm f\) in \(\mathscr S^\dagger \), the first sum converges to
     \[
     \frac12
     \int_{t_1}^{t_2}
       Q_r
     \left(
     S_{\mathcal T-r}\bm f,
     S_{\mathcal T-r}\bm f
     \right)
     \,\mathrm d r
     \]
     as \(N\to\infty\).
     We now analyze the linear part.  Note that
     \begin{align}\label{eq6.1}
     	&\sum_{j=0}^{N-1}
     	\Bigg[
     	\bm Y_{s_{j+1}}(S_{\mathcal T-s_j}\bm f)
     	-
     	\bm Y_{s_j}(S_{\mathcal T-s_j}\bm f)
     	-
     	\int_{s_j}^{s_{j+1}}
     	\bm Y_r(\mathcal A^\dagger S_{\mathcal T-s_j}\bm f)
     	\,\mathrm d r
     	\Bigg]
     	\nonumber\\
     	&=
     	\bm Y_{t_2}(S_{\mathcal T-t_2+\frac{t_2-t_1}{N}}\bm f)
     	-
     	\bm Y_{t_1}(S_{\mathcal T-t_1}\bm f)
     	-
     	\int_{t_1}^{t_1+\frac{t_2-t_1}{N}}
     	\bm Y_r(\mathcal A^\dagger S_{\mathcal T-t_1}\bm f)
     	\,\mathrm d r
     	\nonumber\\
     	&\quad
     	+
     	\sum_{j=1}^{N-1}
     	\Bigg[
     	\bm Y_{s_j}
     	\left(
     	S_{\mathcal T-s_{j-1}}\bm f
     	-
     	S_{\mathcal T-s_j}\bm f
     	\right)
     	-
     	\int_{s_j}^{s_{j+1}}
     	\bm Y_r(\mathcal A^\dagger S_{\mathcal T-s_j}\bm f)
     	\,\mathrm d r
     	\Bigg].
     \end{align}
     By Lemma~\ref{l6.1}, we obtain
     \begin{align*}
     	\bm Y_{s_j}
     	\left(
     	S_{\mathcal{T}-s_{j-1}}\bm f- S_{\mathcal{T}-s_j}\bm f
     	\right)
     	&= \bm Y_{s_j} \left( (s_j-s_{j-1})\mathcal A^\dagger S_{\mathcal{T}-s_j}\bm f
     	+
     	o\left(\frac{t_2-t_1}{N},\mathcal{T}-s_j\right)  \right)\\
     	&= 	\int_{s_j}^{s_{j+1}}
     	\bm Y_{s_j}(\mathcal A^\dagger  S_{\mathcal{T}-s_j}\bm f)\,\mathrm dr 
     	+ \bm Y_{s_j} \left(  o\left(\frac{t_2-t_1}{N},\mathcal{T}-s_j\right)  \right).
     \end{align*}
      By Proposition~\ref{p3.1}, the map $ r\to S_r\bm f $
     is \( \mathcal C^\infty\) as an \(\mathcal S^\dagger\)-valued map. In particular, $ r\to A^\dagger S_r\bm f $ is continuous, and hence uniformly continuous on compact time intervals.
     Moreover, since the limiting field satisfies $ \bm Y\in \mathcal C\big([0,T],(\mathcal S^\dagger)'\big) $, for every \(\bm h\in\mathcal S^\dagger\) the map $ s\to   \bm Y_s ( \bm h ) $ is continuous. Therefore, for fixed \(\bm f\in\mathcal S^\dagger\), the map $ (s,r)\to
       \bm Y_s( A^\dagger S_r\bm f ) $ is uniformly continuous.
      Then the preceding expression implies that the
      sum in \eqref{eq6.1} converges to
      \[
      \bm Y_{t_2}(S_{\mathcal T-t_2}\bm f)
      -
      \bm Y_{t_1}(S_{\mathcal T-t_1}\bm f).
      \]
      Consequently,
      \begin{align*}
      &\lim_{N\to\infty}
      \prod_{j=0}^{N-1}
      X_{s_{j+1}}^{s_j}(S_{\mathcal T-s_j}\bm f) \\
      &=
      \exp \left\lbrace 
      \sum_{j=0}^{N-1} \frac{1}{2} \int_{s_j}^{s_{j+1}}  Q_r(S_{\mathcal{T} -s_j}\bm f,S_{\mathcal{T}-s_j}\bm f)\,\mathrm dr   
      + i \left( \bm Y_{t_2}(S_{\mathcal{T}-t_2}\bm f)- \bm Y_{t_1}(S_{\mathcal{T}-t_1}\bm f) \right)   
      \right\rbrace  
       =\frac{Z_{t_2}}{Z_{t_1}}	 
      \end{align*}
      almost surely. Since the complex exponentials are uniformly bounded on compact time
      intervals, the Dominated Convergence Theorem gives additionally the  \(\mathcal L^1\)  convergence.
      Let \(G\) be a bounded \(\mathcal F_{t_1}\)-measurable random variable. Then
      \[
      \mathbb E
      \left[
      G\frac{Z_{t_2}}{Z_{t_1}}
      \right]
      =
      \lim_{N\to\infty}
      \mathbb E
      \left[
      G
      \prod_{j=0}^{N-1}
      X_{s_{j+1}}^{s_j}(S_{\mathcal T-s_j}\bm f)
      \right].
      \]
      Using the martingale property of each factor
      \(X_{s_{j+1}}^{s_j}(S_{\mathcal T-s_j}\bm f)\), and taking conditional expectations
      successively from \(j=N-1\) down to \(j=0\), we obtain
      \[
      \mathbb E
      \left[
      G
      \prod_{j=0}^{N-1}
      X_{s_{j+1}}^{s_j}(S_{\mathcal T-s_j}\bm f)
      \right]
      =
      \mathbb E[G].
      \]
      Letting \(N\to\infty\), we get
      \[
      \mathbb E
      \left[
      G\frac{Z_{t_2}}{Z_{t_1}}
      \right]
      =
      \mathbb E[G].
      \]
      Therefore \(\{Z_t:0\le t\le\mathcal T\}\) is a martingale.
      
      It follows from the martingale property of \(Z_t\) that, for every \(0\le s<t\le\mathcal T\), $ \mathbb E[Z_t\mid\mathcal F_s]=Z_s $.
      Expanding this identity gives
      \[
      \mathbb E
      \left[
      \exp\left\{
      i\bm Y_t(S_{\mathcal T-t}\bm f)
      \right\}
      \mid\mathcal F_s
      \right]
      =
      \exp\left\{
      -\frac12
      \int_s^t
        Q_r
      \left(
      S_{\mathcal T-r}\bm f,
      S_{\mathcal T-r}\bm f
      \right)
      \,\mathrm d r
      +
      i\bm Y_s(S_{\mathcal T-s}\bm f)
      \right\}.
      \]
      Now set $ \bm g=S_{\mathcal T-t}\bm f $.
      Then
      \[
      S_{\mathcal T-s}\bm f
      =
      S_{t-s}S_{\mathcal T-t}\bm f
      =
      S_{t-s}\bm g.
      \]
      Thus
      \[
      \mathbb E
      \left[
      \exp\left\{
      i\bm Y_t(\bm g)
      \right\}
      \mid\mathcal F_s
      \right]
      =
      \exp\left\{
      -\frac12
      \int_s^t
       Q_r
      \left(
      S_{t-r}\bm g,
      S_{t-r}\bm g
      \right)
      \,\mathrm d r
      +
      i\bm Y_s(S_{t-s}\bm g)
      \right\}.
      \]
      Replacing \(\bm g\) by \(\lambda\bm f\), we obtain
      \[
      \mathbb E
      \left[
      \exp\left\{
      i\lambda\bm Y_t(\bm f)
      \right\}
      \mid\mathcal F_s
      \right]
      =
      \exp\left\{
      -\frac{\lambda^2}{2}
      \int_s^t
        Q_r
      \left(
      S_{t-r}\bm f,
      S_{t-r}\bm f
      \right)
      \,\mathrm d r
      +
      i\lambda\bm Y_s(S_{t-s}\bm f)
      \right\},
      \]
      which means that, conditionally on \(\mathcal F_s\), the random variable
      \(\bm Y_t(\bm f)\) is Gaussian with mean $ \bm Y_s(S_{t-s}\bm f) $
      and variance $ \int_s^t
        Q_r
      \left(
      S_{t-r}\bm f,
      S_{t-r}\bm f
      \right)
      \,\mathrm d r $.
      Since the distribution of the initial field \(\bm Y_0\) is fixed by the covariance
      \(\sigma\), the above conditional characteristic function determines all
      finite-dimensional distributions of the process. Indeed, by applying the conditional
      characteristic function successively to linear combinations of test functions at
      ordered times, one determines the joint characteristic function of $ \left(
      \bm Y_{t_1}(\bm f_1),
      \ldots,
      \bm Y_{t_m}(\bm f_m)
      \right) $
      for arbitrary $ 0\le t_1<\cdots<t_m\le T $ and $ \bm f_1,\ldots,\bm f_m\in\mathscr S^\dagger $.
      Hence the martingale problem admits at most one solution in distribution. This proves
      the uniqueness of the limiting Ornstein-Uhlenbeck process.

\end{proof}

\subsection{Construction of limit points}

Let \(\{\mathbb Q_n\}_{n \in \mathbb{N} }\) be the sequence of probability measures on
\(\mathcal D([0,T],(\mathscr S^\dagger)')\) induced by the fluctuation fields $\left\{\bm Y_t^n:0\le t\le T\right\}_{n \in \mathbb{N} }  $. We will prove that any limit point of \(\{\mathbb Q_n\}_{n \in \mathbb{N} }\)  satisfies the conditions of Proposition \ref{p6.1}.

Fix a time-independent test function $ \bm f=(f_1,f_2)^\top\in\mathscr S^\dagger $. Then
\[
\bm M_t^n(\bm f)
=
\left\langle \bm Y_t^n,\bm f\right\rangle
-
\left\langle \bm Y_0^n,\bm f\right\rangle
-
\int_0^t
(\partial_s+L_n)
\left(
Y_1^{n,s}(f_1)+Y_2^{n,s}(f_2)
\right)
\,\mathrm d s .
\]
By the generator computation in the proof of Theorem~\ref{th2.2}, we have
\[
\begin{aligned}
	\int_0^t
	(\partial_s+L_n)Y_k^{n,s}(f_k)
	\,\mathrm d s
	&=
	\int_0^t
	\Bigg\{
	Y_k^{n,s}(\Delta f_k)
	+
	\gamma_{3-k}Y_{3-k}^{n,s}(f_k)
	-
	\gamma_kY_k^{n,s}(f_k)
	+
	\mathfrak R_k^n(f_k)
	\Bigg\}
	\,\mathrm d s ,
\end{aligned}
\]
where
\[
\begin{aligned}
	\mathfrak R_k^n(f_k)
	&=
	Y_k^{n,s}\left((\Delta_n-\Delta)f_k\right)  
	+
	\frac1{\sqrt n}
	\sum_{x\in\{1,n-1\}}
	\left(
	\gamma_k\bar\eta_k^s(x)
	-
	\gamma_{3-k}\bar\eta_{3-k}^s(x)
	\right)
	f_k\left(\frac{x}{n}\right)
	+
	R_k^n(f_k).
\end{aligned}
\]
Here \(R_k^n(f_k)\) denotes the boundary error defined in \eqref{eq5.1}.

The boundary terms cannot be treated componentwise. However, for the coupled
fluctuation field, the boundary contributions cancel because the test function
\(\bm f\in\mathscr S^\dagger \) satisfies the adjoint Robin boundary conditions. 
More precisely,
using the same Taylor expansion and boundary cancellation as in the proof of
Theorem~\ref{th2.2}, one obtains
 \[
\lim_{n \to \infty} \int_0^t
\left[
\mathfrak R_1^n(f_1)+\mathfrak R_2^n(f_2)
\right]
\,\mathrm d s
\to 0
\]
uniformly for \(t\in[0,T]\).
Therefore,
\[
\begin{aligned}
\bm M_t^n(\bm f)
=
\left\langle \bm Y_t^n,\bm f\right\rangle
-
\left\langle \bm Y_0^n,\bm f\right\rangle
-
\int_0^t
\left\langle
\bm Y_s^n,
\mathcal A^\dagger\bm f
\right\rangle
\,\mathrm d s
+
o_n(1),	 
\end{aligned}
\]
where \(o_n(1)\to0\) in probability.  
 
By Proposition~\ref{p5.1}, the martingales
\(\bm M_t^n(\bm f)\) converge to a centered Gaussian process whose predictable
quadratic variation is
\[
\int_0^t
  Q_s(\bm f,\bm f)
\,\mathrm d s= \Phi_{11} + \Phi_{12} + \Phi_{21} + \Phi_{22}.
\]
Thus, passing to the limit along a convergent subsequence, every limit point
\(\mathbb Q\) is concentrated on paths satisfying the following martingale problem:
for every \(\bm f\in\mathscr S^\dagger \),
\[
W_t(\bm f)
=
\left\langle \bm Y_t,\bm f\right\rangle
-
\left\langle \bm Y_0,\bm f\right\rangle
-
\int_0^t
\left\langle
\bm Y_s,
\mathcal A^\dagger\bm f
\right\rangle
\,\mathrm d s
\] and
\[
\mathcal N_t(\bm f)
=
W_t(\bm f)^2
-
\int_0^t
  Q_s(\bm f,\bm f)
\,\mathrm d s
\]
are martingales.

By Proposition~\ref{p6.1},
the characterization of the limit points obtained above, and the tightness result
proved in Section~7, it remains only to identify the covariance structure of the
limit. This follows from Theorem~\ref{th2.2}. More precisely, the martingale part converges
to a centered Gaussian process with covariance determined by the bilinear form
\(  Q_t\), while the initial field \(\bm Y_0\) has covariance \(\sigma\) and is uncorrelated with the limiting martingale noise. Hence, for every
\(\bm f,\bm g\in\mathscr S^\dagger\),
\[
\mathbb E
\left[
 \bm Y_t ( \bm f )
 \bm Y_s ( \bm g )
\right]
=
\sigma(S_t\bm f,S_s\bm g)
+
\int_0^{s\wedge t}
 Q_r
\left(
S_{t-r}\bm f,
S_{s-r}\bm g
\right)
\,\mathrm d r .
\]
This completes the proof of Theorem~\ref{th2.3}.

\section{Tightness}

In this section we prove the tightness of the sequence of fluctuation fields $ \left\{\bm Y_t^n:\,0\le t\le T\right\}_{n\in \mathbb{N}} $
in \(\mathcal D([0,T],(\mathscr S^\dagger)')\).
By Mitoma's criterion, it is enough to prove that, for every
\(\bm f\in\mathscr S^\dagger \), the real-valued processes $\left\{\bm Y_t^n(\bm f):\,0\le t\le T\right\}_{n\in \mathbb{N} } $ are tight in \(\mathcal D([0,T];\mathbb R)\).  

\begin{proposition}\label{p7.1}
	The space \(\mathscr S^\dagger \), endowed with the seminorms
	\[
	\|\bm f\|_k
	=
	\sup_{u\in[0,1]}
	\left\|
	\partial_u^k\bm f(u)
	\right\|_{\mathbb R^2},
	\qquad k\in\mathbb N\cup\{0\}
	\]
	is a Fr\'echet space.
\end{proposition}

\begin{proof}
	It is well known that \( \mathcal C^\infty([0,1])\), endowed with the seminorms
	\[
	\|g\|_k
	=
	\sup_{u\in[0,1]}
	\left|g^{(k)}(u)\right|,
	\qquad k\in\mathbb N\cup\{0\}
	\]
	is a Fr\'echet space (\cite{ref8}). 
	Since $ \mathcal C^\infty([0,1];\mathbb R^2)
	\simeq
	\mathcal C^\infty([0,1])\times C^\infty([0,1]) $ and a finite product of Fr\'echet spaces is again a Fr\'echet space, it follows that
	\(\mathcal C^\infty([0,1];\mathbb R^2)\) is a Fr\'echet space.
	
	It remains to show that \(\mathscr S^\dagger \) is a closed subspace of
	\(\mathcal C^\infty([0,1];\mathbb R^2)\). For each \(m\ge0\), define the linear maps
	\[
	B_0^{(m)}(\bm f)
	:=
	\partial_u\left((\mathcal A^\dagger)^m\bm f\right)(0)
	-
	K_L^\top\left((\mathcal A^\dagger)^m\bm f\right)(0),
	\]
	and
	\[
	B_1^{(m)}(\bm f)
	:=
	\partial_u\left((\mathcal A^\dagger)^m\bm f\right)(1)
	+
	K_R^\top\left((\mathcal A^\dagger)^m\bm f\right)(1).
	\]
	Then
	\[
	\mathscr S^\dagger 
	=
	\bigcap_{m\ge0}
	\ker B_0^{(m)}
	\cap
	\ker B_1^{(m)}.
	\]
	Since point evaluations of smooth functions and their derivatives are continuous in
	the Fr\'echet topology of \(\mathcal C^\infty([0,1];\mathbb R^2)\), the maps
	\(B_0^{(m)}\) and \(B_1^{(m)}\) are continuous linear maps. Hence their kernels are
	closed. Therefore \(\mathscr S^\dagger \), being a countable intersection of closed subspaces,
	is closed in \(\mathcal C^\infty([0,1];\mathbb R^2)\). Consequently, \(\mathscr S^\dagger \) is a
	Fr\'echet space.
\end{proof}

We shall use the following form of Mitoma's criterion.

\begin{proposition}[Mitoma's criterion, \cite{ref6}]
	\label{p 7.2}
	Let \(\{X^n_t:\,0\le t\le T\}_{n\in \mathbb{N} }\) be a sequence of processes with values in
	\(\mathcal D([0,T],(\mathscr S^\dagger)')\). Then the sequence is tight in
	\(\mathcal D([0,T],(\mathscr S^\dagger)')\) if and only if, for every \(\bm f\in\mathscr S^\dagger \), the
	real-valued processes $ \left\{X^n_t(\bm f):\,0\le t\le T\right\}_{n \in \mathbb{N} } $
	are tight in \(\mathcal D([0,T], \mathbb R)\).
\end{proposition}

Thus it remains to prove the tightness of $ \left\{\bm Y_t^n(\bm f):\,0\le t\le T\right\}_{n\in \mathbb{N}} $ in \(\mathcal D([0,T],\mathbb R)\), for every fixed \(\bm f\in\mathscr S^\dagger\). We shall use Aldous' criterion.

\begin{proposition}[Aldous' criterion, \cite{ref7}]
	\label{p7.3}
	A sequence of real-valued processes $ \left\{X_t^n: 0\le t\le T\right\}_{n\in \mathbb{N}} $
	is tight in \(\mathcal D([0,T];\mathbb R)\) if the following two conditions hold:
	\begin{enumerate}
		\item[(i)] $ \lim_{M\to\infty}
		\limsup_{n\to\infty}
		\mathbb P_{\mu^n}^n
		\left(
		\sup_{0\le t\le T}|X_t^n|>M
		\right)
		=0 $.
		
		\item[(ii)]
		For every \(\varepsilon>0\),
		\[
		\lim_{\delta\to 0}
		\limsup_{n\to\infty}
		\sup_{0<\lambda\le\delta}
		\sup_{\tau\in\mathcal T_T}
		\mathbb P_{\mu^n}^n
		\left(
		|X_{\tau+\lambda}^n-X_\tau^n|>\varepsilon
		\right)
		=0,
		\]
		where \(\mathcal T_T\) denotes the family of stopping times bounded by \(T\).
	\end{enumerate}
\end{proposition}

Fix \(\bm f=(f_1,f_2)^\top\in\mathscr S^\dagger \). By the martingale decomposition, it is
enough to prove tightness separately for the initial term, the martingale term, and
the integral term:
\[
\left\{\bm Y_0^n(\bm f)\right\}_{n \in \mathbb{N} }, ~ ~
\left\{
M_t^n(\bm f):\,0\le t\le T
\right\}_{n \in \mathbb{N}}, ~~
\left\{
\int_0^t
(\partial_s+L_n)\bm Y_s^n(\bm f)
\,\mathrm d s:\,0\le t\le T
\right\}_{n \in \mathbb{N} }.
\]
 
For the initial term, by definition,
\[
\bm Y_0^n(\bm f)
=
\frac1{\sqrt n}
\sum_{x=1}^{n-1}
\bar\eta_1^0(x)
f_1\left(\frac{x}{n}\right)
+
\frac1{\sqrt n}
\sum_{x=1}^{n-1}
\bar\eta_2^0(x)
f_2\left(\frac{x}{n}\right).
\]
A direct computation gives
\[
\begin{aligned}
	\mathbb E_{\mu^n}^n
	\left[
	\left(\bm Y_0^n(\bm f)\right)^2
	\right]
	&=
	\frac1n
	\sum_{k=1}^2
	\sum_{x=1}^{n-1}
	f_k\left(\frac{x}{n}\right)^2
	\chi\left(\rho_k^{n,0}(x)\right)       
	-
	\frac2n
	\sum_{x=1}^{n-1}
	f_1\left(\frac{x}{n}\right)
	f_2\left(\frac{x}{n}\right)
	\rho_1^{n,0}(x)\rho_2^{n,0}(x)          \\
	&\quad
	+
	\frac2n
	\sum_{k=1}^2
	\sum_{x<y}
	f_k\left(\frac{x}{n}\right)
	f_k\left(\frac{y}{n}\right)
	\varphi_{kk}^{n,0}(x,y)                  
	+
	\frac2n
	\sum_{x\ne y}
	f_1\left(\frac{x}{n}\right)
	f_2\left(\frac{y}{n}\right)
	\varphi_{12}^{n,0}(x,y).
\end{aligned}
\]
Here
\[
\chi(r)=r(1-r),
\]
and
\[
\varphi_{ij}^{n,0}(x,y)
=
\mathbb E_{\mu^n}^n
\left[
\left(\eta_i^0(x)-\rho_i^{n,0}(x)\right)
\left(\eta_j^0(y)-\rho_j^{n,0}(y)\right)
\right].
\]
By the assumptions on the initial profile and the initial two-point correlations, we have
\[
\sup_{n\ge1}
\mathbb E_{\mu^n}^n
\left[
\left(\bm Y_0^n(\bm f)\right)^2
\right]
\le C.
\]
Therefore the sequence \(\{\bm Y_0^n(\bm f)\}_{n\in \mathbb{N}}\) is tight.
 
For the martingale term,
by Proposition~\ref{p5.1}, the sequence $ \left\{
M_t^n(\bm f):\,0\le t\le T
\right\}_{n\in \mathbb{N} } $
is tight in \(\mathcal D([0,T],\mathbb R)\).

For the integral term,
since \(\bm f\in\mathscr S^\dagger \), the generator computation yields
\[
(\partial_s+L_n)\bm Y_s^n(\bm f)
=
\bm Y_s^n(\mathcal A^\dagger\bm f)
+
R_n^s(\bm f),
\]
where \(R_n^s(\bm f)\) collects the discretization errors and the boundary remainders.
By the adjoint Robin boundary conditions imposed in the definition of \(\mathscr S^\dagger \),
these boundary remainders cancel in the coupled field, and there exists a constant
\(C>0\), independent of \(n\), such that
\[
|R_n^s(\bm f)|
\le
\frac{C}{\sqrt n},
\qquad 0\le s\le T.
\]
Therefore,
\[
\begin{aligned}
	 \mathbb E_{\mu^n}^n
	\left[
	\sup_{0\le t\le T}
	\left(
	\int_0^t
	(\partial_s+L_n)\bm Y_s^n(\bm f)
	\,\mathrm d s
	\right)^2
	\right]                                       
	&\le
	T
	\int_0^T
	\mathbb E_{\mu^n}^n
	\left[
	\left(
	(\partial_s+L_n)\bm Y_s^n(\bm f)
	\right)^2
	\right]
	\,\mathrm d s                                \\
	&\le
	2T
	\int_0^T
	\mathbb E_{\mu^n}^n
	\left[
	\left(
	\bm Y_s^n(\mathcal A^\dagger\bm f)
	\right)^2
	\right]
	\,\mathrm d s
	+
	\frac{2CT^2}{n}.
\end{aligned}
\]

Let $ \bm g
=
\mathcal A^\dagger\bm f
=
(g_1,g_2)^\top $.
Then
\[
g_1
=
\Delta f_1-\gamma_1f_1+\gamma_1f_2,
\qquad
g_2
=
\Delta f_2+\gamma_2f_1-\gamma_2f_2.
\]
Since \(\bm f\in\mathscr S^\dagger \), we also have \(\bm g\in\mathscr S^\dagger\), and in particular
\(\bm g\) is smooth and bounded. 
 For this \(\bm g\),
 \[
 \bm Y_s^n(\bm g)
 =
 \frac1{\sqrt n}
 \sum_{x=1}^{n-1}
 \bar\eta_1^s(x)
 g_1\left(\frac{x}{n}\right)
 +
 \frac1{\sqrt n}
 \sum_{x=1}^{n-1}
 \bar\eta_2^s(x)
 g_2\left(\frac{x}{n}\right).
 \]
 As in the estimate of the initial term, we obtain
 \[
 \begin{aligned}
 	\mathbb E_{\mu^n}^n
 	\left[
 	\left(\bm Y_s^n(\bm g)\right)^2
 	\right]
 	&=
 	\frac1n
 	\sum_{k=1}^2
 	\sum_{x=1}^{n-1}
 	g_k\left(\frac{x}{n}\right)^2
 	\chi\left(\rho_k^{n,s}(x)\right)             
 	-
 	\frac2n
 	\sum_{x=1}^{n-1}
 	g_1\left(\frac{x}{n}\right)
 	g_2\left(\frac{x}{n}\right)
 	\rho_1^{n,s}(x)\rho_2^{n,s}(x)              \\
 	&\quad
 	+
 	\frac1n
 	\sum_{k=1}^2
 	\sum_{x\ne y}
 	g_k\left(\frac{x}{n}\right)
 	g_k\left(\frac{y}{n}\right)
 	\varphi_{kk}^{n,s}(x,y)                     
 	+
 	\frac2n
 	\sum_{x\ne y}
 	g_1\left(\frac{x}{n}\right)
 	g_2\left(\frac{y}{n}\right)
 	\varphi_{12}^{n,s}(x,y).
 \end{aligned}
 \]
By the two-point correlation estimate established later in
Proposition~\ref{p8.3}, we have
\[
\mathbb E_{\mu^n}^n
\left[
\sup_{0\le t\le T}
\left(
\int_0^t
(\partial_s+L_n)\bm Y_s^n(\bm f)
\,\mathrm d s
\right)^2
\right]
\le C.
\]
Therefore, the first condition of the Aldous' criterion holds.

We now verify Aldous' second condition. 
Let \(\tau\in\mathcal T_T\), \(0<\theta\le\delta\), and set
\[
\tau_\theta=(\tau+\theta)\wedge T.
\]
By the martingale decomposition, we have
\[
\bm Y_{\tau_\theta}^n(\bm f)-\bm Y_\tau^n(\bm f)
=
\bm M_{\tau_\theta}^n(\bm f)-\bm M_\tau^n(\bm f)
+
\int_\tau^{\tau_\theta}
(\partial_s+L_n)\bm Y_s^n(\bm f)
\,\mathrm d s .
\]
Therefore, by Chebyshev's inequality,
\[
\begin{aligned}
	\mathbb P_{\mu^n}^n
	\left(
	\left|
	\bm Y_{\tau_\theta}^n(\bm f)-\bm Y_\tau^n(\bm f)
	\right|>\varepsilon
	\right)
	&\le
	\frac{2}{\varepsilon^2}
	\mathbb E_{\mu^n}^n
	\left[
	\left(
	\bm M_{\tau_\theta}^n(\bm f)-\bm M_\tau^n(\bm f)
	\right)^2
	\right]  \\
	&\quad
	+
	\frac{2}{\varepsilon^2}
	\mathbb E_{\mu^n}^n
	\left[
	\left(
	\int_\tau^{\tau_\theta}
	(\partial_s+L_n)\bm Y_s^n(\bm f)
	\,\mathrm d s
	\right)^2
	\right].
\end{aligned}
\]
We estimate the two terms separately.
First, for the martingale part, 
\[
\mathbb E_{\mu^n}^n
\left[
\left(
\bm M_{\tau_\theta}^n(\bm f)-\bm M_\tau^n(\bm f)
\right)^2
\right]
=
\mathbb E_{\mu^n}^n
\left[
\left\langle \bm M^n(\bm f)\right\rangle_{\tau_\theta}
-
\left\langle \bm M^n(\bm f)\right\rangle_\tau
\right].
\]
Moreover,
\[
\left\langle\bm M^n(\bm f)\right\rangle_t
=
\int_0^t
\Gamma_n^s (\bm f,\bm f )
\,\mathrm d s.
\]
From the explicit expression of the carr\'e du champ, the boundedness of the occupation
variables, and the smoothness of \(\bm f\), there exists a constant
\(C>0\), independent of \(n\), such that
\[
0\le \Gamma_n^s (\bm f,\bm f)\le C,
\qquad 0\le s\le T.
\]
Hence
\[
\begin{aligned}
	\mathbb E_{\mu^n}^n
	\left[
	\left(
	\bm M_{\tau_\theta}^n(\bm f)-\bm M_\tau^n(\bm f)
	\right)^2
	\right]
	=
	\mathbb E_{\mu^n}^n
	\left[
	\int_\tau^{\tau_\theta}
	\Gamma_n^s (\bm f,\bm f)
	\,\mathrm d s
	\right]  
	\le
	C\theta
	\le
	C\delta.
\end{aligned}
\]
For the integral term, by the uniform second moment bound obtained above,
\begin{align*}
	\mathbb E_{\mu^n}^n
	\left[
	\left(
	\int_\tau^{\tau_\theta}
	(\partial_s+L_n)\bm Y_s^n(\bm f)
	\,\mathrm d s
	\right)^2
	\right] 
	\le
	\theta\,
	\mathbb E_{\mu^n}^n
	\left[
	\int_\tau^{\tau_\theta}
	\left|(\partial_s+L_n)\bm Y_s^n(\bm f)\right|^2
	\,\mathrm d s
	\right]  
	\le  C\delta.
\end{align*}
Combining the two estimates, we obtain
\[
\lim_{\delta\to 0}
\limsup_{n\to\infty}
\sup_{\tau\in\mathcal T_T}
\sup_{0<\theta\le\delta}
\mathbb P_{\mu^n}^n
\left(
\left|
\bm Y_{\tau_\theta}^n(\bm f)-\bm Y_\tau^n(\bm f)
\right|
>\varepsilon
\right)
=0.
\]
This proves Aldous' second condition.

Hence, for every \(\bm f\in\mathscr S^\dagger \), the sequence $ \left\{
\bm Y_t^n(\bm f):\,0\le t\le T
\right\}_{n\in \mathbb{N} } $
is tight in \(\mathcal D([0,T], \mathbb R)\). By Mitoma's criterion, it follows that $ \left\{
\bm Y_t^n:\,0\le t\le T
\right\}_{n\in \mathbb{N}} $ is tight in \(\mathcal D([0,T],(\mathscr S^\dagger)')\).

\section{Discrete equations and two-point correlation estimates}

In this section we derive the discrete equations satisfied by the two-point
correlation functions of the two-species system. These equations provide the key spatial correlation estimates and second-moment bounds required for the tightness argument.

\subsection{The discrete equation}

Let  
\[
\Sigma_n=\{1,\ldots,n-1\},
\qquad
V_n=\{(x,y):0<x<y<n\}.
\]
For \(x,y\in\Sigma_n\), \(x<y\), and \(t\in[0,T]\), define the two-point correlation
functions by
\[
\varphi_{ij}^{n,t}(x,y)
=
\mathbb E_{\mu^n}^n
\left[
\left(\eta_i^t(x)-\rho_i^{n,t}(x)\right)
\left(\eta_j^t(y)-\rho_j^{n,t}(y)\right)
\right],
\qquad i,j\in\{1,2\},
\]
where $ \rho_i^{n,t}(x) = \mathbb E_{\mu^n}^n\left[\eta_i^t(x)\right] $.
Equivalently,
\[
\varphi_{ij}^{n,t}(x,y)
=
\mathbb E_{\mu^n}^n
\left[
\eta_i^t(x)\eta_j^t(y)
\right]
-
\rho_i^{n,t}(x)\rho_j^{n,t}(y).
\]
We impose the boundary convention $\varphi_{ij}^{n,t}(x,y)=0  $ whenever \(x=0\) or \(y=n\). We also write
\[
\bm\varphi^{n,t}(x,y)
=
\left(
\varphi_{11}^{n,t}(x,y),
\varphi_{12}^{n,t}(x,y),
\varphi_{21}^{n,t}(x,y),
\varphi_{22}^{n,t}(x,y)
\right)^\top .
\]

We now introduce the discrete operator acting on the spatial variables. Let
\(\mathscr A_n\) be the generator of the nearest-neighbor random walk on \(V_n\),
acting on functions \(F:V_n\to\mathbb R\) by
\[
(\mathscr A_nF)(u)
=
\sum_{v\in V_n}
c_n(u,v)\bigl(F(v)-F(u)\bigr),
\qquad u\in V_n,
\]
where
\[
c_n(u,v)
=
\begin{cases}
	1, & \text{if } \|u-v\| =1,\quad u,v\in V_n,\\
	0, & \text{otherwise}.
\end{cases}
\]

The additional effect of stirring on the diagonal \(y=x+1\) is described by the matrix
\[
\mathbf P
=
\begin{pmatrix}
	0 & 0 & 0 & 0\\
	0 & -1 & 1 & 0\\
	0 & 1 & -1 & 0\\
	0 & 0 & 0 & 0
\end{pmatrix}.
\]
This term exchanges the labels in the components \(\varphi_{12}\) and
\(\varphi_{21}\), while leaving \(\varphi_{11}\) and \(\varphi_{22}\) unchanged.

The bulk conversion at the first coordinate \(x\) is represented by
\[
\mathbf R_x
= \mathbf 1_{\{2\le x\le n-2\}} \left( M \otimes I_2 \right) 
= \mathbf 1_{\{2\le x\le n-2\}}
\begin{pmatrix}
	-\gamma_1 & 0 & \gamma_2 & 0\\
	0 & -\gamma_1 & 0 & \gamma_2\\
	\gamma_1 & 0 & -\gamma_2 & 0\\
	0 & \gamma_1 & 0 & -\gamma_2
\end{pmatrix},
\]
and the bulk conversion at the second coordinate \(y\) is represented by
\[
\mathbf R_y
= \mathbf 1_{\{2\le y\le n-2\}} \left( I_2\otimes M \right)  
= \mathbf 1_{\{2\le y\le n-2\}}
\begin{pmatrix}
	-\gamma_1 & \gamma_2 & 0 & 0\\
	\gamma_1 & -\gamma_2 & 0 & 0\\
	0 & 0 & -\gamma_1 & \gamma_2\\
	0 & 0 & \gamma_1 & -\gamma_2
\end{pmatrix}.
\]

The left boundary contribution, which acts when \(x=1\), is given by
\[
\mathbf B_x^n
=
n\mathbf 1_{\{x=1\}}
\begin{pmatrix}
	-(1-\alpha_2) & 0 & -\alpha_1 & 0\\
	0 & -(1-\alpha_2) & 0 & -\alpha_1\\
	-\alpha_2 & 0 & -(1-\alpha_1) & 0\\
	0 & -\alpha_2 & 0 & -(1-\alpha_1)
\end{pmatrix}.
\]
Similarly, the right boundary contribution, which acts when \(y=n-1\), is
\[
\mathbf B_y^n
=
n\mathbf 1_{\{y=n-1\}}
\begin{pmatrix}
	-(1-\beta_2) & -\beta_1 & 0 & 0\\
	-\beta_2 & -(1-\beta_1) & 0 & 0\\
	0 & 0 & -(1-\beta_2) & -\beta_1\\
	0 & 0 & -\beta_2 & -(1-\beta_1)
\end{pmatrix}.
\]

Finally, the inhomogeneous source term is supported on the diagonal boundary
\(\{y=x+1\}\). Define
\[
\mathbf H^{n,t}(x,y)
=
\left(
H_{11}^{n,t}(x,y),
H_{12}^{n,t}(x,y),
H_{21}^{n,t}(x,y),
H_{22}^{n,t}(x,y)
\right)^\top,
\]
where
\[
H_{ij}^{n,t}(x,y)
=
-n^2
\left(
\rho_i^{n,t}(x+1)-\rho_i^{n,t}(x)
\right)
\left(
\rho_j^{n,t}(x+1)-\rho_j^{n,t}(x)
\right)
\mathbf 1_{\{y=x+1\}} .
\]

From the Kolmogorov forward equation, we can give the discrete equation for the two-point correlations.
\begin{proposition}\label{p8.1}
	For every \((x,y)\in V_n\), the vector-valued two-point correlation function
	\(\bm\varphi^{n,t}\) satisfies
	\begin{equation}\label{eq8.1}
		\begin{aligned}
			\partial_t\bm\varphi^{n,t}(x,y)
			&=
			n^2(\mathscr A_n \mathrm {I}_4 \bm\varphi^{n,t})(x,y)
			+
			n^2\mathbf P\bm\varphi^{n,t}(x,y)\mathbf 1_{\{y=x+1\}}  \\
			&\quad
			+
			\mathbf R_x\bm\varphi^{n,t}(x,y)
			+
			\mathbf R_y\bm\varphi^{n,t}(x,y)   
			+
			\mathbf B_x^n\bm\varphi^{n,t}(x,y)
			+
			\mathbf B_y^n\bm\varphi^{n,t}(x,y)
			+
			\mathbf H^{n,t}(x,y).
		\end{aligned}
	\end{equation}
\end{proposition}

 \subsection{The nine-component discrete equation}
 
 The four-component correlation system derived above is not convenient for the
 correlation estimates. Indeed, the boundary matrices in the four-component system
 contain negative off-diagonal entries, and therefore cannot be regarded as
 Markov generators. To recover a closed Markovian structure, we enlarge the state space from four to nine components by including the vacancy as an additional state.
 Let $ E=\{0,1,2\} $, where \(0\) denotes a vacant site, while \(1\) and \(2\) denote the two particle species. For \(a,b\in E\), \(0<x<y<n\), and \(t\in[0,T]\), define
 \[
 \varphi_{ab}^{n,t}(x,y)
 =
 \mathbb E_{\mu^n}^n
 \left[
 \eta_a^t(x)\eta_b^t(y)
 \right]
 -
 \rho_a^{n,t}(x)\rho_b^{n,t}(y).
 \]
 We impose the boundary convention $ \varphi_{ab}^{n,t}(x,y)=0 $ whenever \(x=0\) or \(y=n\). We write
 \[
 \begin{aligned}
 	\bm\varphi^{n,t}(x,y)
 	=
 	\big(
 	&\varphi_{00}^{n,t}(x,y),
 	\varphi_{01}^{n,t}(x,y),
 	\varphi_{02}^{n,t}(x,y),
 	\varphi_{10}^{n,t}(x,y),
 	\varphi_{11}^{n,t}(x,y),  \\
 	&\varphi_{12}^{n,t}(x,y),
 	\varphi_{20}^{n,t}(x,y),
 	\varphi_{21}^{n,t}(x,y),
 	\varphi_{22}^{n,t}(x,y)
 	\big)^\top .
 \end{aligned}
 \]
Since
\[
\sum_{k=0}^2\rho_k^{n,t}(x)=1,
\qquad
\sum_{k=0}^2\eta_k^x(t)=1, \quad \forall x \in \Sigma_{n},
\]
we have, for \(i=0, 1,2\),
\[
\varphi_{0i}^{n,t}(x,y)
=
-\varphi_{1i}^{n,t}(x,y)-\varphi_{2i}^{n,t}(x,y), \qquad 
\varphi_{i0}^{n,t}(x,y)
=
-\varphi_{i1}^{n,t}(x,y)-\varphi_{i2}^{n,t}(x,y).
\]
More generally,
\[
\bm\varphi^{n,t}(x,y)\in\mathcal H,
\]
where
\[
\mathcal H
=
\left\{
v=(v_{ab})_{a,b\in E}\in\mathbb R^9:
\sum_{a=0}^2v_{ab}=0\ \text{for every } b\in E,\quad
\sum_{b=0}^2v_{ab}=0\ \text{for every } a\in E
\right\}.
\]

 We now introduce the matrices entering the nine-component equation. The diagonal
 exchange matrix is
 \[
 \mathsf P
 =
 \begin{pmatrix}
 	0&0&0&0&0&0&0&0&0\\
 	0&-1&0&1&0&0&0&0&0\\
 	0&0&-1&0&0&0&1&0&0\\
 	0&1&0&-1&0&0&0&0&0\\
 	0&0&0&0&0&0&0&0&0\\
 	0&0&0&0&0&-1&0&1&0\\
 	0&0&1&0&0&0&-1&0&0\\
 	0&0&0&0&0&1&0&-1&0\\
 	0&0&0&0&0&0&0&0&0
 \end{pmatrix}.
 \]
 This matrix appears only on the discrete diagonal \(y=x+1\), and represents the
 exchange of the two labels at neighboring sites.
 
 The bulk conversion matrix is
 \[
 Q
 =
 \begin{pmatrix}
 	0&0&0\\
 	0&-\gamma_1&\gamma_2\\
 	0&\gamma_1&-\gamma_2
 \end{pmatrix}.
 \]
 Thus, conversion at the first coordinate \(x\) is represented by $ Q\otimes \mathrm I_3 $, whereas conversion at the second coordinate \(y\) is represented by $ \mathrm I_3\otimes Q $.
 
The left and right boundary matrices are respectively
\[
R^\alpha
=
\begin{pmatrix}
	-(\alpha_1+\alpha_2) & \alpha_0 & \alpha_0\\
	\alpha_1 & -\alpha_0 & 0\\
	\alpha_2 & 0 & -\alpha_0
\end{pmatrix},
\qquad
\alpha_0=1-\alpha_1-\alpha_2,
\]
and
\[
R^\beta
=
\begin{pmatrix}
	-(\beta_1+\beta_2) & \beta_0 & \beta_0\\
	\beta_1 & -\beta_0 & 0\\
	\beta_2 & 0 & -\beta_0
\end{pmatrix},
\qquad
\beta_0=1-\beta_1-\beta_2.
\]
Accordingly, the boundary action at the left coordinate is described by $ R^\alpha\otimes \mathrm I_3 $, and the boundary action at the right coordinate by $ \mathrm I_3\otimes R^\beta  $.

We now state the nine-component correlation equation. 
\begin{proposition}[Nine-component two-point correlation equation]
	\label{p8.2}
	For every \((x,y)\in V_n\), the nine-component correlation function
	\(\bm\varphi^{n,t}(x,y)\) satisfies
	\begin{equation}\label{eq8.2}
		\partial_t\bm\varphi^{n,t}(x,y)
		=
		\mathcal L_n^\ast\bm\varphi^{n,t}(x,y)
		+
		\bm G^{n,t}(x,y),
	\end{equation}
	where
	\begin{align}\label{eq8.3}
		\mathcal L_n^\ast
		={}&
		n^2\mathscr A_n \mathrm I_9
		+
		n^2\mathbf 1_{\{y=x+1\}}\mathsf P
		+
		\mathbf 1_{\{2\le x\le n-2\}}(Q\otimes \mathrm I_3)
		\nonumber\\
		&+
		\mathbf 1_{\{2\le y\le n-2\}}(\mathrm I_3\otimes Q)
		+
		n\mathbf 1_{\{x=1\}}(R^\alpha\otimes \mathrm I_3)
		+
		n\mathbf 1_{\{y=n-1\}}(\mathrm I_3\otimes R^\beta).
	\end{align}
	The source term \(\bm G^{n,t}\) is supported on the discrete diagonal $ D_n=\{(x,y)\in V_n:\ y=x+1\} $ and for \(a,b\in E\),
	\begin{equation}\label{eq8.4}
		G_{ab}^{n,t}(x,y)
		=
		-n^2
		\left(
		\rho_a^{n,t}(x+1)-\rho_a^{n,t}(x)
		\right)
		\left(
		\rho_b^{n,t}(x+1)-\rho_b^{n,t}(x)
		\right)
		\mathbf 1_{\{y=x+1\}}.
	\end{equation}
\end{proposition}
 
 \begin{remark}
 	Since $ \sum_{a=0}^2 \left( \rho_a^{n,t}(x+1)-\rho_a^{n,t}(x) \right) =0 $, we have
 	\[
 	\sum_{a=0}^2G_{ab}^{n,t}(x,y)=0,
 	\qquad
 	\sum_{b=0}^2G_{ab}^{n,t}(x,y)=0.
 	\]
 	Hence $ \bm G^{n,t}(x,y)\in\mathcal H $.
 Moreover, the subspace \(\mathcal H\) is invariant under the homogeneous evolution
 generated by \(\mathcal L_n^\ast\). Indeed, the spatial operator \(\mathscr A_n\)
 acts componentwise and therefore preserves the defining constraints of
 \(\mathcal H\). In addition, a direct verification from the explicit matrices shows
 that each matrix appearing in \(\mathcal L_n^\ast\) maps \(\mathcal H\) into
 itself. Hence \(\mathcal L_n^\ast\) maps \(\mathcal H\)-valued functions into
 \(\mathcal H\)-valued functions. 
 Since $ \bm\varphi^{n,0}(x,y)\in\mathcal H $ and $ \bm G^{n,t}(x,y)\in\mathcal H $, Duhamel's formula implies that
 \[
 \bm\varphi^{n,t}(x,y)\in\mathcal H,
 \qquad t\ge0 .
 \]
 \end{remark}
 
 Since the homogeneous operator \(\mathcal L_n^*\) has nonnegative off-diagonal entries and zero column sums, it can be identified with the Kolmogorov forward operator of a finite-state
continuous-time Markov chain on $ \mathcal E_n=V_n\times E^2 $.
Equivalently, $ \mathcal L_n=(\mathcal L_n^\ast)^\top $ is the backward generator acting on test functions \(F:\mathcal E_n\to\mathbb R\),
namely
\[
(\mathcal L_nF)(z)
=
\sum_{z'\neq z}
c_n(z,z')\bigl(F(z')-F(z)\bigr).
\]
The jump rates are as follows.

First, the spatial motion is given by nearest-neighbor jumps in \(V_n\). If
\((r,s)\in V_n\) and $ \|(r,s)-(x,y)\| =1$,
then
\[
((x,y),(a,b))
\longrightarrow
((r,s),(a,b))
\]
occurs at rate \(n^2\).

 Second, if \(y=x+1\), the diagonal color exchange is
 \[
 ((x,x+1),(a,b))
 \longrightarrow
 ((x,x+1),(b,a))
 \]
 at rate \(n^2\).
 
 Third, if \(2\le x\le n-2\), the bulk conversion at the first coordinate is
 \[
 ((x,y),(1,b))
 \longrightarrow
 ((x,y),(2,b))
 \]
 at rate \(\gamma_1\), and
 \[
 ((x,y),(2,b))
 \longrightarrow
 ((x,y),(1,b))
 \]
 at rate \(\gamma_2\). The conversion at the second coordinate \(y\) is defined in the
 same way.

 Fourth, if \(x=1\), the left boundary transitions are
 \[
 ((1,y),(0,b))
 \longrightarrow
 ((1,y),(1,b))
 \]
 at rate \(n\alpha_1\),
 \[
 ((1,y),(0,b))
 \longrightarrow
 ((1,y),(2,b))
 \]
 at rate \(n\alpha_2\),
 \[
 ((1,y),(1,b))
 \longrightarrow
 ((1,y),(0,b))
 \]
 at rate \(n\alpha_0\), and
 \[
 ((1,y),(2,b))
 \longrightarrow
 ((1,y),(0,b))
 \]
 at rate \(n\alpha_0\).
 
 Fifth, if \(y=n-1\), the right boundary transitions are
 \[
 ((x,n-1),(a,0))
 \longrightarrow
 ((x,n-1),(a,1))
 \]
 at rate \(n\beta_1\),
 \[
 ((x,n-1),(a,0))
 \longrightarrow
 ((x,n-1),(a,2))
 \]
 at rate \(n\beta_2\),
 \[
 ((x,n-1),(a,1))
 \longrightarrow
 ((x,n-1),(a,0))
 \]
 at rate \(n\beta_0\), and
 \[
 ((x,n-1),(a,2))
 \longrightarrow
 ((x,n-1),(a,0))
 \]
 at rate \(n\beta_0\).

 We denote by $P_t^n(z,z')
 =
 \mathbb P_z(Z_t^n=z') $
 the transition probability of this Markov chain.

\subsection{Two-point correlation estimate}

\begin{proposition}\label{p8.3}
     There exists a constant \(C>0\) such that
	\[
	\sup_{  t\ge 0}
	\max_{1\le x<y\le n-1}
	|\varphi_{ij}^{n,t}(x,y)|
	\le
	\frac{C}{n},
	\qquad i,j=1,2.
	\]
\end{proposition}

\begin{proof}
	Since \(\mathcal L_n^\ast\) generates the forward evolution of the nine-component
	system, we denote by \((S_t^n)_{t\ge0}\) the corresponding semigroup. By Duhamel's
	formula,
	\begin{equation}\label{eq8.5}
		\bm\varphi^{n,t}
		=
		S_t^n\bm\varphi^{n,0}
		+
		\int_0^t
		S_{t-s}^n\bm G^{n,s}\,\mathrm d s.
	\end{equation}
	We first establish a domination estimate for \(S_t^n\). For
	\(v=(v_{ab})_{a,b\in E}\in\mathbb R^9\), set $ \|v\|_1=\sum_{a,b\in E}|v_{ab}| $.
	The diagonal exchange matrix can be written as $ \mathsf P=\Pi- \mathrm I_9 $, where \(\Pi= \mathsf P+ \mathrm I_9\) satisfies $ \Pi^2 = \mathrm I_9 $.
 	Hence, for \(s\ge0\),
	\[
	e^{s\mathsf P}
	=
	e^{-s}e^{s\Pi}
	= 
	e^{-s}
	\left(
	\sum_{k=0}^\infty \frac{s^{2k}}{(2k)!}I
	+
	\sum_{k=0}^\infty \frac{s^{2k+1}}{(2k+1)!}\Pi
	\right)
	=\frac{1+e^{-2s}}2I_9
	+
	\frac{1-e^{-2s}}2\Pi.
	\]
	Since \(\Pi\) is an isometry for the \(\ell^1\)-norm,
	\[
	\|e^{s\mathsf P}v\|_1\le \|v\|_1, \qquad \|e^{sn^2\mathsf P}v\|_1\le \|v\|_1.
	\]
	
	Next, the matrix \(Q\) has nonnegative off-diagonal entries and zero column sums.
	Therefore \(e^{sQ}\) is nonnegative and has column sums equal to one. Hence
	\[
	\|e^{sQ}w\|_1
	\le
	\sum_{a=0}^2\sum_{b=0}^2(e^{sQ})_{ab}|w_b|
	=
	\sum_{b=0}^2|w_b|
	=
	\|w\|_1.
	\]
	Consequently,
	\[
	\|e^{s(Q\otimes \mathrm I_3)}v\|_1\le \|v\|_1,
	\qquad
	\|e^{s(\mathrm I_3\otimes Q)}v\|_1\le \|v\|_1.
	\]
	
	It remains to consider the boundary matrices. Since \(R^\alpha\) and \(R^\beta\) have
	nonnegative off-diagonal entries and zero column sums, they generate finite-state
	Markov semigroups in the forward convention. On the zero-sum subspace
	\[
	\mathcal H_0
	=
	\left\{
	w=(w_0,w_1,w_2)\in\mathbb R^3:
	\sum_{a=0}^2w_a=0
	\right\},
	\]
	the eigenvalue \(0\) is removed, and the remaining spectrum lies strictly in the left
	half-plane. Hence there exist constants \(\kappa_\alpha,\kappa_\beta>0\) such that
	\[
	\|e^{sR^\alpha}w\|_1
	\le
	e^{-\kappa_\alpha s}\|w\|_1,
	\qquad
	\|e^{sR^\beta}w\|_1
	\le
	e^{-\kappa_\beta s}\|w\|_1,
	\qquad
	w\in\mathcal H_0.
	\]
	Therefore, on \(\mathcal H\),
	\[
	\|e^{s(R^\alpha\otimes I_3)}v\|_1
	\le
	 e^{-\kappa_\alpha s}\|v\|_1, \qquad 
	\|e^{s(I_3\otimes R^\beta)}v\|_1
	\le
	 e^{-\kappa_\beta s}\|v\|_1.
	\]
	 
	Using the Trotter product formula, the \(\ell^1\)-contraction of the interior color
	dynamics, and the strict dissipation of the boundary color dynamics on the zero-sum
	subspace, we obtain the following domination estimate. There exists \(C>0\) such that,
	for every \(\mathcal H\)-valued function \(H:V_n\to\mathcal H\),
	\begin{equation}\label{eq8.6}
		\|(S_t^nH)(x,y)\|_1
		\le
		C
		\left(
		e^{tn^2\widehat A_n}\|H\|_1
		\right)(x,y),
		\qquad (x,y)\in V_n.
	\end{equation}
	Here \(\widehat A_n\) is the scalar generator of a nearest-neighbor walk on
	\(V_n\cup\partial V_n\) with absorbing boundary rates. More precisely, for
	\(h:V_n\cup\partial V_n\to\mathbb R\),
	\[
	(\widehat A_nh)(u)
	=
	\sum_{v\in V_n\cup\partial V_n}
	\widehat c_n(u,v)\big(h(v)-h(u)\big),
	\qquad u\in V_n,
	\]
	where
	\[
	\widehat c_n(u,v)
	=
	\begin{cases}
		1, & \|u-v\|_1=1,\quad u,v\in V_n,\\[0.1cm]
		\dfrac{\lambda}{n}, & \|u-v\|_1=1,\quad u\in V_n,\ v\in\partial V_n,\\[0.2cm]
		0, & \text{otherwise},
	\end{cases}
	\]
	with $ \lambda=\min\{\kappa_\alpha,\kappa_\beta\} $.
	
	We now estimate two terms in \eqref{eq8.5}. First, by the assumption on
	the initial correlations and the identities defining the zero-sum extension, the full
	nine-component initial field satisfies
	\[
	\|\bm\varphi^{n,0}\|_\infty\le \frac{C}{n}.
	\]
	The domination estimate gives
	\[
	\|S_t^n\bm\varphi^{n,0}\|_\infty
	\le
	C\|\bm\varphi^{n,0}\|_\infty
	\le
	\frac{C}{n}.
	\]
	
	Next, by the one-point density gradient estimate, 
	we have
	\[
	\|\bm G^{n,t}(x,y)\|_1
	\le
	C\mathbf 1_{\{y=x+1\}}.
	\]
	Indeed, by Proposition \ref{p8.4}, each discrete gradient
	\[
	\rho_a^{n,t}(x+1)-\rho_a^{n,t}(x)
	\]
	is \(O(n^{-1})\), and the prefactor \(n^2\) in
	\eqref{eq8.4} gives an \(O(1)\) bound on the diagonal.
	
	Combining this with \eqref{eq8.6}, we obtain
	\[
	\begin{aligned}
		\left\|
		\int_0^t
		S_{t-s}^n\bm G^{n,s}\,\mathrm d s
		\right\|_1
		&\le
		C\int_0^t
		\left(
		e^{(t-s)n^2\widehat A_n}\mathbf 1_{D_n}
		\right)(x,y)\,\mathrm d s      \\
		&=
		\frac{C}{n^2}
		\int_0^{tn^2}
		\mathbb P_{(x,y)}(X_r\in D_n)\,\mathrm d r,
	\end{aligned}
	\]
	where \(X_r\) is the scalar walk generated by \(\widehat A_n\). By the diagonal
	occupation-time estimate for the corresponding one-species problem (\cite[Lemma 3.2]{ref2}),
	\[
	\sup_{(x,y)\in V_n}
	\mathbb E_{(x,y)}
	\left[
	\int_0^\infty
	\mathbf 1_{\{X_r\in D_n\}}\,\mathrm d r
	\right]
	\le
	Cn.
	\]
	Therefore,
	\[
	\sup_{  t\ge 0}
	\max_{(x,y)\in V_n}
	\left\|
	\int_0^t
	S_{t-s}^n\bm G^{n,s}\,\mathrm d s
	\right\|_1
	\le
	\frac{C}{n}.
	\]
	Together with the estimate of the initial term, Duhamel's formula gives
	\[
	\sup_{ t\ge 0}
	\max_{(x,y)\in V_n}
	\|\bm\varphi^{n,t}(x,y)\|_1
	\le
	\frac{C}{n}.
	\]
	In particular,
	\[
	\sup_{ t\ge 0}
	\max_{1\le x<y\le n-1}
	|\varphi_{ij}^{n,t}(x,y)|
	\le
	\frac{C}{n},
	\qquad i,j=1,2.
	\]
	This proves the proposition.

\end{proof}

\subsection{One-point density gradient estimate}

We now prove the one-point density gradient estimate used above.

\begin{proposition} 
	\label{p8.4}
	There exists a constant \(C>0\) such that, for \(i=1,2\) and all $ x \in \{ 1, \cdots, n-2 \} $,
	\[
	\left|
	\rho_i^{n,t}(x+1)-\rho_i^{n,t}(x)
	\right|
	\le
	\frac{C}{n}
	\]
	uniformly in $ t \ge 0 $.
	Consequently, the same estimate holds for \(i=0\).
\end{proposition}

\begin{proof}
	We work with the three-component density vector $ \bm\rho_t^n(x)
	=
	\bigl(
	\rho_0^{n,t}(x),
	\rho_1^{n,t}(x),
	\rho_2^{n,t}(x)
	\bigr)^\top  $. It satisfies the discrete linear system
	\begin{equation}\label{eq8.7}
		\begin{cases}
			\partial_t\bm\rho_t^n(x)
			=
			n^2\left(
			\bm\rho_t^n(x+1)+\bm\rho_t^n(x-1)-2\bm\rho_t^n(x)
			\right)
			+
			Q\bm\rho_t^n(x),
			& 2\le x\le n-2,\\[0.15cm]
			\partial_t\bm\rho_t^n(1)
			=
			n^2\left(
			\bm\rho_t^n(2)-\bm\rho_t^n(1)
			\right)
			+
			nR^\alpha\bm\rho_t^n(1),
			& t\ge0,\\[0.15cm]
			\partial_t\bm\rho_t^n(n-1)
			=
			n^2\left(
			\bm\rho_t^n(n-2)-\bm\rho_t^n(n-1)
			\right)
			+
			nR^\beta\bm\rho_t^n(n-1),
			& t\ge0,
		\end{cases}
	\end{equation}
	where $ \bm\alpha=(\alpha_0,\alpha_1,\alpha_2)^\top $, $ \bm\beta=(\beta_0,\beta_1,\beta_2)^\top $ and $ Q, R^\alpha, R^\beta$ are as defined in Section 8.2.
	For \(1\le x\le n-2\), define the rescaled discrete gradient
	\[
	\bm D_t^n(x)
	=
	n\bigl(\bm\rho_t^n(x+1)-\bm\rho_t^n(x)\bigr).
	\]
	We also introduce the boundary gradients
	\[
	\bm D_t^n(0)
	=
	-R^\alpha\bm\rho_t^n(1)
	+
	\frac1n Q\bm\rho_t^n(1),
	\]
	and
	\[
	\bm D_t^n(n-1)
	=
	R^\beta\bm\rho_t^n(n-1)
	-
	\frac1n Q\bm\rho_t^n(n-1).
	\]
	Since \(\bm\rho_t^n(x) \in [0,1]^3 \) for every \(x\) and
	\(t\), and since \(Q,R^\alpha,R^\beta\) are fixed matrices, there exists a
	constant \(C>0\), independent of \(n\) and \(t\), such that
	\[
	\sup_{t\ge0}
	\bigl(
	\|\bm D_t^n(0)\|_1 +\|\bm D_t^n(n-1)\|_1
	\bigr)
	\le C .
	\]
	
	With this convention, equation \eqref{eq8.7} can be rewritten as
	\[
	\partial_t\bm\rho_t^n(x)
	=
	n\bigl(\bm D_t^n(x)- \bm D_t^n(x-1)\bigr)
	+
	Q\bm\rho_t^n(x),
	\qquad 1\le x\le n-1.
	\]
	The definitions of \(\bm D_t^n(0)\) and \( \bm D_t^n(n-1)\) encode the two boundary
	equations. Therefore, for \(1\le x\le n-2\),
	\[
	\partial_t \bm D_t^n(x)
	=
	n^2\bigl(\bm D_t^n(x+1)+\bm D_t^n(x-1) -2\bm D_t^n(x)\bigr)
	+
	Q \bm D_t^n(x),
	\]
	where \(\bm D_t^n(0)\) and \(\bm D_t^n(n-1)\) are regarded as prescribed boundary
	values.
	
	Let $ E_n=\{1,\cdots,n-2\} $.
	Denote by \(L_n^{\mathrm D}\) the Dirichlet discrete Laplacian on \( E_n\),
	\[
	(L_n^{\mathrm D}h)(x)
	=
	n^2\bigl(h(x+1)-2h(x)+h(x-1)\bigr),
	\qquad x\in E_n,
	\]
	with zero boundary $ h(0)=h(n-1)=0 $.
	Then, on \( E_n\), the equation for \( \bm D_t^n\) can be rewritten as
	\[
	\begin{aligned}
		\partial_t \bm D_t^n
		=
		(L_n^{\mathrm D}+Q) \bm D_t^n 
		+
		n^2 \bm D_t^n(0)\mathbf 1_{\{x=1\}}
		+
		n^2 \bm D_t^n(n-1)\mathbf 1_{\{x=n-2\}} .
	\end{aligned}
	\]
	By the variation of constants formula,
	\begin{align}\label{eq8.8}
		\bm D_t^n
		 =
		e^{t(L_n^{\mathrm D}+Q)} \bm D_0^n 
		+
		\int_0^t
		e^{(t-s)(L_n^{\mathrm D}+Q)}
		\left[
		n^2 \bm D_s^n(0)\mathbf 1_{\{x=1\}}
		+
		n^2 \bm D_s^n(n-1)\mathbf 1_{\{x=n-2\}}
		\right]\,\mathrm ds .
	\end{align}
	
	Since \(L_n^{\mathrm D}\) acts on the spatial variable and \(Q\) acts on the
	color variable, the two operators commute. Hence
	\[
	e^{t(L_n^{\mathrm D}+Q)}
	=
	e^{tL_n^{\mathrm D}}e^{tQ}.
	\]
	As already shown in the proof of Proposition \ref{p8.3}, we have
	\[
	\|e^{tQ}v\|_1\le \|v\|_1,
	\qquad v\in\mathbb R^3,\quad t\ge0.
	\]
	The scalar semigroup \(e^{tL_n^{\mathrm D}}\) is positive and sub-Markovian.
	Therefore, for every vector-valued function \(H:E_n\to\mathbb R^3\),
	\[
	\left\|
	\left(e^{t(L_n^{\mathrm D}+Q)}H\right)(x)
	\right\|_1
	\le
	\left(e^{tL_n^{\mathrm D}}\|H\|_1\right)(x)
	\le \max_{y\in E_n}\|H(y)\|_1.
	\]
	Applying this estimate to \eqref{eq8.8} gives
	\[
	\begin{aligned}
		\|\bm D_t^n(x)\|_1
		&\le
		\max_{y\in I_n}\|\bm D_0^n(y)\|_1 
		+
		\int_0^t
		n^2
		\left(e^{(t-s)L_n^{\mathrm D}}\mathbf 1_{\{1\}}\right)(x)
		\|\bm D_s^n(0)\|_1\,\mathrm ds \\
		&\quad
		+
		\int_0^t
		n^2
		\left(e^{(t-s)L_n^{\mathrm D}}\mathbf 1_{\{n-2\}}\right)(x)
		\|\bm D_s^n(n-1)\|_1\,\mathrm ds .
	\end{aligned}
	\]
	It remains to control the two boundary source terms. Define
	\[
	w_0(x)
	=
	\int_0^\infty
	n^2
	\left(e^{rL_n^{\mathrm D}}\mathbf 1_{\{1\}}\right)(x)\,\mathrm dr,
	\qquad x\in E_n.
	\]
	Then \(w_0\) solves the discrete elliptic problem
	\[
	-L_n^{\mathrm D}w_0
	=
	n^2\mathbf 1_{\{1\}},
	\qquad x\in E_n,
	\]
	with zero Dirichlet boundary $ w_0(0)=w_0(n-1)=0 $.
	Equivalently,
	\[
	-\bigl(w_0(x+1)+w_0(x-1)-2w_0(x)\bigr)
	=
	\mathbf 1_{\{x=1\}},
	\qquad x\in E_n.
	\]
	With the Dirichlet convention $ w_0(0)=w_0(n-1)=0 $, the solution on $ E_n $ is given by $ w_0(x)=1-\frac{x}{n-1} $.
	Thus
	\[
	0\le w_0(x)\le 1,
	\qquad x\in E_n.
	\]
	Similarly,
	\[
	w_{n-1}(x)
	:=
	\int_0^\infty
	n^2
	\left(e^{rL_n^{\mathrm D}}\mathbf 1_{\{n-2\}}\right)(x)\,\mathrm dr
	=
	\frac{x}{n-1},
	\qquad x\in E_n,
	\]
	with boundary convention \(w_{n-1}(0)=w_{n-1}(n-1)=0\), and hence
	\[
	0\le w_{n-1}(x)\le1,
	\qquad x\in E_n.
	\]
	Consequently,
	\[
	\int_0^t
	n^2
	\left(e^{(t-s)L_n^{\mathrm D}}\mathbf 1_{\{1\}}\right)(x)\, \mathrm ds
	\le
	w_0(x)
	\le1,
	\]
	and
	\[
	\int_0^t
	n^2
	\left(e^{(t-s)L_n^{\mathrm D}}\mathbf 1_{\{n-2\}}\right)(x)\,\mathrm ds
	\le
	w_{n-1}(x)
	\le1.
	\]
	Therefore,
	\[
	\|\bm D_t^n(x)\|_1
	\le
	\max_{y\in E_n}\|\bm D_0^n(y)\|_1
	+
	\sup_{t\ge0}\|\bm D_t^n(0)\|_1
	+
	\sup_{t\ge0}\|\bm D_t^n(n-1)\|_1 .
	\]
	Taking the maximum over \(x\in E_n\), we obtain
	\[
	\max_{1\le x\le n-2}\|\bm D_t^n(x)\|_1
	\le
	\max_{1\le y\le n-2}\|\bm D_0^n(y)\|_1
	+
	\sup_{t\ge0}\|\bm D_t^n(0)\|_1
	+
	\sup_{t\ge0}\|\bm D_t^n(n-1)\|_1 .
	\]
	By the initial discrete assumption \ref{ass3},
	\[
	\max_{1\le y\le n-2}\|\bm D_0^n(y)\|_1
	=
	n
	\max_{1\le y\le n-2}
	\|\bm\rho_0^n(y+1)-\bm\rho_0^n(y)\|_1
	\le C .
	\]
	Together with the uniform boundedness of the boundary gradients, this
	implies
	\[
	\sup_{t\ge0}
	\max_{1\le x\le n-2}\|\bm D_t^n(x)\|_1
	\le C .
	\]
	Finally, by the definition of \( \bm D_t^n\),
	\[
	\sup_{t\ge0}
	\max_{1\le x\le n-2}
	\|\bm\rho_t^n(x+1)-\bm\rho_t^n(x)\|_1
	\le
	\frac{C}{n}.
	\]
	In particular, for each component \(i=0,1,2\),
	\[
	\sup_{t\ge0}
	\max_{1\le x\le n-2}
	|\rho_i^{n,t}(x+1)-\rho_i^{n,t}(x)|
	\le
	\frac{C}{n}.
	\]
	This completes the proof.
\end{proof}

\bibliographystyle{plain}
\bibliography{ref} 

\end{document}